\documentclass[a4paper,12pt]{amsart}
\usepackage{tikz}

\numberwithin{equation}{section}

\newtheorem{thm}{Theorem}[section]
\newtheorem{prop}[thm]{Proposition}
\newtheorem{lem}[thm]{Lemma}
\newtheorem{cor}[thm]{Corollary}

\theoremstyle{definition}
\theoremstyle{remark}
\newtheorem{rem}[thm]{Remark}
\newtheorem{ex}[thm]{Example}

\allowdisplaybreaks

\newcommand{\NN}{\mathbb{N}}
\newcommand{\RR}{\mathbb{R}}
\newcommand{\TT}{\mathbb{T}}
\newcommand{\QQ}{\mathbb{Q}}

\def\F{\mathcal{F}}

\def\Oo{\mathcal{O}}
\newcommand\mc{\mathrm{mc}}
\def\Cc{\mathcal{C}}

\def\T{\mathcal{T}}

\newcommand{\QE}{\operatorname{QE}}

\newcommand{\Aut}{\operatorname{Aut}}

\newcommand{\lsp}{\operatorname{span}}
\newcommand{\clsp}{\overline{\lsp}}

\def\ebetareg{E^0_{\beta\text{{-reg}}}}
\def\ebetacri{E^0_{\beta\text{{-crit}}}}

\title[KMS states on graph algebras]{\boldmath{KMS states on the $C^*$-algebras of reducible graphs}}

\author[A. an Huef]{Astrid an Huef}
\author[M. Laca]{Marcelo Laca}
\author[I. Raeburn]{Iain Raeburn}
\author[A. Sims]{Aidan Sims}

\address{Astrid an Huef and Iain Raeburn\\ 
Department of Mathematics and Statistics\\University of Otago\\PO Box 56\\Dunedin 9054\\New Zealand}
\email{astrid@maths.otago.ac.nz, iraeburn@maths.otago.ac.nz}

\address{Marcelo Laca\\ Department of Mathematics and Statistics\\
University of Victoria\\
Victoria, BC V8W 3P4\\
Canada}
\email{laca@math.uvic.ca}

\address{Aidan  Sims\\ School of Mathematics and Applied Statistics\\
University of Wollongong\\NSW 2522\\Australia}
\email{asims@uow.edu.au}

\date{21 February 2014, with minor revisions 29 April 2014}

\thanks{This research has been supported by the Marsden Fund of the Royal Society of New Zealand, the Natural Sciences and Engineering Research Council of Canada, and the Australian Research Council. The authors are grateful for the opportunity to work on this project at the Banff International Research Station in November 2013.}

\begin{document}
\begin{abstract}
We consider the dynamics on the $C^*$-algebras of finite graphs obtained by lifting the gauge action to an action of the real line. Enomoto, Fujii and Watatani proved that if the vertex matrix of the graph is irreducible, then the dynamics on the graph algebra admits a single KMS state. We have previously studied the dynamics on the Toeplitz algebra, and explicitly described a finite-dimensional simplex of KMS states for inverse temperatures above a critical value. Here we study the KMS states for graphs with reducible vertex matrix, and for inverse temperatures at and below the critical value. We prove a general result which describes all the KMS states at a fixed inverse temperature, and then apply this theorem to a variety of examples. We find that there can be many patterns of phase transition, depending on the behaviour of paths in the underlying graph.
\end{abstract}

\maketitle

\section{Introduction}

Composing the gauge action of $\TT$ with the map $t\mapsto e^{it}$ gives a natural dynamics on any Cuntz-Krieger algebra or graph algebra. Enomoto, Fujii and Watatani  proved thirty years ago that for a simple Cuntz-Krieger algebra $\Oo_A$, this dynamics admits a unique KMS state, and that this state has  inverse temperature the natural logarithm $\ln\rho(A)$ of the spectral radius $\rho(A)$ (which is also the Perron-Frobenius eigenvalue of $A$) \cite{EFW}. Recently Kajiwara and Watatani revisited this question for the $C^*$-algebras of finite graphs with sources, and found many more KMS states \cite{KW}. Other authors are currently interested in KMS states on the $C^*$-algebras of infinite graphs \cite{T, CL} or on the $C^*$-algebras of higher-rank graphs \cite{Y,aHLRSk}.

We recently studied KMS states on the Toeplitz algebra $\T C^*(E)$ of a finite graph $E$~\cite{aHLRS1}. For inverse temperatures $\beta$ larger than a critical value $\beta_c$, we described a simplex of KMS$_\beta$ states whose dimension is determined by the number of vertices in the graph \cite[Theorem~3.1]{aHLRS1}. This gave a concrete implementation of an earlier result of Exel and Laca \cite[Theorem~18.4]{EL}, at least as it applies to the gauge dynamics. The critical inverse temperature $\beta_c$ in \cite{aHLRS1} is $\ln\rho(A)$ where $A$ is the vertex matrix of the graph $E$. When $A$ is irreducible in the sense of Perron-Frobenius theory (and in particular if $C^*(E)$ is simple), we showed that there is a unique KMS$_{\ln\rho(A)}$ state on $\T C^*(E)$, and that this state factors through $C^*(E)$. 

Here we consider a finite graph $E$ whose vertex matrix $A$ is reducible, and aim to find all the KMS states on $\T C^*(E)$ and $C^*(E)$. We have organised our results so that we can describe the KMS states at each fixed inverse temperature. From \cite[Theorem~3.1]{aHLRS1}, we already have a concrete description of the simplex of KMS$_\beta$ states on $\T C^*(E)$ for $\beta>\ln\rho(A)$, and we know exactly which ones factor through $C^*(E)$ \cite[Corollary~6.1]{aHLRS1}.

Our first main theorem concerns the critical value $\beta=\ln \rho(A)$ (Theorem~\ref{KMScrit}). It identifies two different families of extreme KMS$_{\ln \rho(A)}$ states. The first family $\{\psi_C\}$ is parametrised by a set of strongly connected components $C$ of $E$ such that the matrix $A_C:=A|_{C\times C}$ satisfies $\beta=\ln \rho(A_C)$ (in the theorem we say exactly which components belong to this set). The states $\psi_C$ all factor through $C^*(E)$. Then we consider the hereditary closure $H$ in $E^0$ of the components $C$ with $\beta=\ln \rho(A_C)$, and the complementary graph $E\backslash H$ with vertex set $E^0\backslash H$. The second family $\{\phi_{v}\}$ of extremal KMS$_{\ln \rho(A)}$ consists of states which factor through a natural quotient map of $\T C^*(E)$ onto $\T C^*(E\backslash H)$ (see Proposition~\ref{quotmapH}), and which is parametrised by $E^0\backslash H$. The convex hull of $\{\psi_C\}\cup\{\phi_v\}$ is the full simplex of KMS$_{\ln \rho(A)}$ states. The proof of Theorem~\ref{KMScrit} involves some rather intricate computations using the Perron-Frobenius theory for the matrices $A_C$. 

In \S\ref{sec:allKMS}, we describe the KMS$_{\ln\rho(A)}$ states for a fixed inverse temperature $\beta$ satisfying $\beta<\ln\rho(A)$. In Theorem~\ref{thm:altogether}, we consider the hereditary closure $H_\beta$ of the connected components $C$ with $\ln\rho(A_C)>\beta$. If $\beta>\ln\rho(A_{E^0\backslash H_\beta})$, the KMS$_{\beta}$ states all factor through the quotient $\T C^*(E\backslash H_\beta)$, and an application of \cite[Theorem~3.1]{aHLRS1} gives a concrete description of these states. If $\beta=\ln\rho(A_{E^0\backslash H_\beta})$, then applying Theorem~\ref{KMScrit} to $E\backslash H_\beta$ shows that there are two families $\{\psi_C\}$ and $\{\phi_{v}\}$ of extremal KMS$_\beta$ states. Theorem~\ref{thm:altogether} also identifies the states which factor through $C^*(E)$, where there are some tricky subtleties involving the saturations of the sets $H_\beta$ and $K_\beta$. 

By applying Theorem~\ref{thm:altogether} as $\beta$ decreases, we can in principle find all KMS states on $\T C^*(E)$ and $C^*(E)$ for every finite graph $E$. In \S\ref{sec:exs}, we show how this works on a variety of examples, and find in particular that there are graphs for which our dynamics has many phase transitions. These examples shed considerable light on the possible behaviour of KMS states, and in particular on what happens between the various critical inverse temperatures discussed in \cite[\S14]{EL}. We close with a section of concluding remarks in which we discuss the range of possible inverse temperatures, and the connections with the results of \cite{EL, CL}.

\section{Background}

\subsection{Directed graphs and their Toeplitz algebras} Suppose that $E=(E^0,E^1,r,s)$ is a directed graph. We use the conventions of \cite{CBMS} for paths, so that, for example, $ef$ is a path when $s(e)=r(f)$. We write $E^n$ for the set of paths of length $n$, and $E^*:=\bigcup_{n\in \NN}E^n$. For vertices $v,w$, we write $vE^nw$ for the set $\{\mu\in E^n: r(\mu)=v\text{ and }s(\mu)=w\}$ (and we allow variations on this theme).

A Toeplitz-Cuntz-Krieger family $(P,S)$ consists of mutually orthogonal projections $\{P_v:v\in E^0\}$ and partial isometries $\{S_e:e\in E^1\}$ such that $S_e^*S_e=P_{s(e)}$ for every $e\in E^1$ and 
\begin{equation}\label{TCK}
P_v\geq \sum_{e\in F}S_eS_e^*\text{ for every $v\in E^0$ and finite subset $F$ of $vE^1=r^{-1}(v)$.}
\end{equation} 
Here we consider only finite graphs, and then it suffices to impose the inequality \eqref{TCK} for $F=vE^1$. 
The Toeplitz algebra $\T C^*(E)$ is generated by a universal Toeplitz-Cuntz-Krieger family $(p,s)$; the existence of such an algebra was proved in \cite[Theorem~4.1]{FR}. For $\mu\in E^n$, we define $s_\mu:=s_{\mu_1}s_{\mu_2}\cdots s_{\mu_n}$. Then each $s_\mu$ is also a partial isometry, and we have
\[
\T C^*(E):=\clsp\big\{s_\mu s_\nu^*:\mu,\nu\in E^*,\;s(\mu)=s(\nu)\big\}.
\]

We shall work mostly in the Toeplitz algebra $\T C^*(E)$ rather the usual graph algebra $C^*(E)$, and it is therefore convenient to view  $C^*(E)$ as the quotient of $\T C^*(E)$ by the ideal generated by 
\[
\Big\{ p_v-\sum_{r(e)=v}s_es_e^*:v\in E^0\Big\}.
\]
We write $\pi_E$ for the quotient map of $\T C^*(E)$ onto $C^*(E)$, and $\bar p_v:=\pi_E(p_v)$, $\bar s_e:=\pi_E(s_e)$. The pair $(\bar p,\bar s)$ is then universal for Cuntz-Krieger families in the usual way.

\subsection{Ideals in Toeplitz algebras}\label{idealsinT}

We are interested in graphs whose $C^*$-algebras $C^*(E)$ are not simple. The standard theory (as in \cite{KPRR}, \cite{BPRS} or \cite[\S4]{CBMS}) says that ideals in $C^*(E)$ are determined by subsets $H$ of $E^0$ which are both hereditary ($v\in H$ and $vE^*w\not=\emptyset$ imply $w\in H$) and saturated ($s(vE^1)\subset H$ implies $v\in H$). In the Toeplitz algebra, there are more ideals, and in particular every hereditary subset determines one. We need to know what the quotient is.

\begin{prop}\label{quotmapH}
Suppose that $H$ is a hereditary set of vertices in a directed graph $E$ and that $H$ is not all of $E^0$. Then $E\backslash H:=(E^0\backslash H, s^{-1}(E^0\backslash H),r,s)$ is a directed graph, and there is a homomorphism $q_H:\T C^*(E)\to \T C^*(E\backslash H)=C^*(p^{E\backslash H}, s^{E\backslash H})$ such that 
\begin{equation}\label{defTCKquot} 
q_H(p_v)=\begin{cases}
p^{E\backslash H}_v&\text{if $v\in E^0\backslash H$}\\
0&\text{if $v\in H$,}
\end{cases}
\quad\text{and}\quad
q_H(s_e)=\begin{cases}
s^{E\backslash H}_e&\text{if $s(e)\in E^0\backslash H$}\\
0&\text{if $s(e)\in H$.}
\end{cases}
\end{equation}
The homomorphism is surjective, and its kernel is the ideal $J_H$ generated by $\{p_v:v\in H\}$.
\end{prop}

\begin{proof}
Since $s$ maps $(E\backslash H)^1:=s^{-1}(E^0\backslash H)$ to $(E\backslash H)^0:=E^0\backslash H$, and since $r(e)\in H$ implies $s(e)\in H$, $r$ maps $(E\backslash H)^1$ into $(E\backslash H)^0$ also. Thus $E\backslash H$ is a directed graph. The formulas on the right-hand sides of \eqref{defTCKquot} define a Toeplitz-Cuntz-Krieger $E$-family in $\T C^*(E\backslash H)$, and hence the universal property of $\T C^*(E)$ gives the existence of the homomorphism $q_H$. It is surjective because its range contains all the generators of $\T C^*(E\backslash H)$. The kernel of $q_H$ contains all the generators of $J_H$, so $J_H\subset \ker q_H$, and hence $q_H$ factors through the quotient map $q:\T C^*(E)\to \T C^*(E)/ J_H$. We write $\bar q_H$ for the homomorphism on $\T C^*(E)/ J_H$ such that $q_H=\bar q_H\circ q$.

To see that $J_H$ is all of $\ker q_H$, we construct a left inverse for $\bar q_H$. A quick check shows that the elements $\{q(p_v),q(s_e):v\in E^0\backslash H,\; e\in s^{-1}(E^0\backslash H)\}$ form a Toeplitz-Cuntz-Krieger $(E\backslash H)$-family in $\T C^*(E)/ J_H$. (It is crucial that we are not trying to impose a Cuntz-Krieger relation at vertices in $E^0\backslash H$ which receive edges from $H$.) Thus there is a homomorphism  $\rho:\T C^*(E\backslash H)\to \T C^*(E)/J_H$ such that $\rho(p^{E\backslash H}_v)=q(p_v)$ and $\rho(s^{E\backslash H}_e)=q(s_e)$. Since $s(e)\in H$ implies that $q(s_e)=0$, the range of $\rho$ contains the images of all the generators of $\T C^*(E)$, and hence $\rho$ is surjective. A quick check shows that $\bar q_H\circ \rho$ fixes the generators of $\T C^*(E\backslash H)$, and hence is the identity on $\T C^*(E\backslash H)$. Now the surjectivity of $\rho$ implies that $\rho\circ\bar q_H$ is the identity on $\T C^*(E\backslash H)$, so $\bar q_H$ is injective, and we have $\ker q_H=J_H$.
\end{proof}

\subsection{Decompositions of the vertex matrix}

Let $E$ be a finite directed graph. The vertex matrix of $E$ is the $E^0\times E^0$ matrix $A$ with entries $A(v,w)=|vE^1w|$; the powers of $A$ then have entries $A^n(v,w)=|vE^nw|$. We will do computations using block decompositions of the vertex matrix $A$. For subsets $C,D\subset E^0$, we write $A_{C,D}$ for the $C\times D$ subblock of $A$, and $A_C:=A_{C,C}$. We usually choose decompositions of $E^0=C_1\sqcup C_2\sqcup\cdots\sqcup C_n$ such that the associated block decomposition of $A$ is upper-triangular.

For $v,w\in E^0$, we write $v\leq w\Longleftrightarrow vE^*w\not=\emptyset$, and $v\sim w\Longleftrightarrow v\leq w \text{ and }w\leq v$. It is easy to check that $\sim$ is an equivalence relation on $E^0$ (we have $v\sim v$ for all $v\in E^0$ because $E^0\subset E^*$). We write $E^0/\!\!\sim$ for the set of equivalence classes, and refer to these equivalence classes as the \emph{strongly connected components} of $E$.  When $C\in E^0/\!\!\sim$, the matrix $A_C$ is either a $1\times 1$ zero matrix (if $C=\{v\}$ is a single vertex with no loops, in which case we say $C$ is a trivial component), or an irreducible matrix in the sense of Perron-Frobenius theory (so that for every $v,w\in C$, there exists $n$ such that $A^n(v,w)>0$). 

We next order the vertex set $E^0$ to ensure that the vertex matrix takes a convenient block upper-triangular form. The relation $\leq $ descends to a well-defined partial order on $E^0/\!\!\sim$; when $C\leq D$, we say that $D$ talks to $C$. We list first the trivial components for which $A_C=(0)$ and which do not talk to nontrivial components; we list them in an order such that $w$ appears after $v$ when $v\leq w$. Next we list the components which are minimal for the order $\leq$ on the remaining components, grouping the vertices in the same component together. Then we list the trivial components which talk only to the components we have listed so far, and so on. This decomposes $A$ as a block upper-triangular matrix in which the diagonal components $A_C$ are either $1\times 1$ zero matrices or are irreducible. We will refer to such a decomposition as a \emph{Seneta decomposition} of $A$. (Though since Seneta uses different conventions in \cite[\S1.2]{Seneta}, the decomposition he discusses there is a block lower-triangular matrix and our minimal components would become maximal\footnote{Unfortunately, there is no universal convention as to whether $A(v,w)$ should refer to edges from $w$ to $v$ or edges from $v$ to $w$. Our
convention arises from viewing directed edges as arrows in a category, in
which case one expects $ef:=e\circ f$ to have source $s(f)$ and range
$r(e)$. This convention is standard in many places: for example, in the
substantial literature on higher-rank graphs, which have strong links to
higher-dimensional subshifts \cite{KP, PRW}, and in studying equivalences
for categories of modules over path algebras of quivers \cite{S1, S2}, see
especially the discussion in \cite[\S5.4]{S1}.}.)

\subsection{KMS states} 

We denote the gauge actions of $\TT$ on $\T C^*(E)$ and $C^*(E)$ by $\gamma$. We are interested in the dynamics $\alpha$ given, on both $\T C^*(E)$ and $C^*(E)$, by $\alpha_t=\gamma_{e^{it}}$. For KMS states, we use the conventions of our previous paper \cite{aHLRS1}. Thus we know from \cite[Proposition~2.1]{aHLRS1} that a state $\phi$ of $\T C^*(E)$ is a KMS$_\beta$ state for some $\beta\in \RR$ if and only if
\begin{equation}\label{workhorse}
\phi(s_\mu s_\nu^*)=\delta_{\mu,\nu}e^{-\beta|\mu|}\phi(p_{s(\mu)})\quad\text{for all $\mu,\nu\in E^*$.} 
\end{equation}
For fixed $\beta$ the KMS$_\beta$ states on $(\T C^*(E),\alpha)$ form a simplex, which we shall refer to as the \emph{KMS$_\beta$ simplex} of $(\T C^*(E),\alpha)$. The KMS$_0$ states are the invariant traces.

Since the results of \cite[\S3]{aHLRS1} already describe all the KMS states for large inverse temperatures, we do not have anything new to say about ground states or KMS$_\infty$ states.

\section{KMS states and quotients}

When the vertex matrix $A$ of $E$ is irreducible, there are no KMS$_\beta$ states on the Toeplitz algebra $\T C^*(E)$ when $\beta<\ln \rho(A)$. So it seems reasonable that if $C$ is a strongly connected component with $\ln \rho(A_C)>\beta$, then every KMS$_\beta$ state must vanish on vertex projections $p_v$ with $v\in C$. The key to our analysis of reducible graphs is that KMS states must also vanish on any projections $p_v$ for vertices $v$ that connect to such components $C$. The next result makes this precise.

\begin{prop}\label{old2.5}
Suppose that $H$ is a hereditary subset of $E^0$, and $q_H:\T C^*(E)\to \T C^*(E\backslash H)$ is the surjection of Proposition~\ref{quotmapH}. Then for every $\beta\in [0,\infty)$, $q_H^*:\psi\mapsto \psi\circ q_H$ is an affine injection of the KMS$_\beta$ simplex of $(\T C^*(E\backslash H),\alpha)$ into the KMS$_\beta$ simplex of $(\T C^*(E),\alpha)$. If $\{C\in H/\!\!\sim:\ln\rho(A_C)>\beta\}$ generates $H$ as a hereditary subset of $E^0$, then $\phi(p_v)=0$ for every KMS$_\beta$ state $\phi$ on $\T C^*(E)$ and every $v\in H$; if in addition $H$ is not all of $E^0$, then $q_H^*$ is surjective.
\end{prop}

\begin{lem}\label{Lemfactor}
Suppose that $H$ is the hereditary subset of $E^0$ generated by $\Cc_1\subset E^0/\!\!\sim$, and that $\beta\leq \ln \rho(A_C)$ for all $C\in \Cc_1$. Suppose that $\phi$ is a KMS$_\beta$ state on $(\T C^*(E),\alpha)$, and that $v$ belongs to the complement of $\bigcup\{C\in\Cc_1:\beta=\ln\rho(A_C)\}$ in $H$. Then $\phi(p_v)=0$.
\end{lem}

When $\{C\in H/\!\!\sim:\ln\rho(A_C)>\beta\}$ generates $H$, as in Proposition~\ref{old2.5}, Lemma~\ref{Lemfactor} applies to every $v\in H$ with $\Cc_1=\{C\in H/\!\!\sim:\ln\rho(A_C)>\beta\}$. The extra generality in the Lemma will be useful in the proof of Proposition~\ref{discardbottom} below.

\begin{proof}
For every path $\mu$ with $s(\mu)=v$, \eqref{TCK} implies that $p_{r(\mu)}\geq s_\mu s_\mu^*$, and hence
\begin{equation}\label{estphipv}
0\leq \phi(p_v)=\phi(s_\mu^* s_\mu)=e^{\beta|\mu|}\phi(s_\mu s_\mu^*)\leq e^{\beta|\mu|}\phi(p_{r(\mu)}).
\end{equation}
Proposition~2.1(c) of \cite{aHLRS1}  implies that the vector $m^\phi:=(\phi(p_w))$ in $[0,1]^{E^0}$ satisfies the subinvariance relation $Am^\phi\leq e^\beta m^\phi$, and for every $C\in \Cc_1$ we have
\begin{equation}\label{locsubinv}
A_C(m^\phi|_C)\leq A_C(m^\phi|_C)+A_{C,H\backslash C}(m^\phi|_{H\backslash C})=(Am^\phi)|_C\leq e^\beta m^\phi|_C.
\end{equation}

Since $v\in H$ and $H$ is generated by $\Cc_1$, there exists $C\in \Cc_1$ such that $CE^*v\not=\emptyset$. Then either $\beta<\ln\rho(A_C)$ or $\beta=\ln\rho(A_C)$. Suppose that $\beta<\ln\rho(A_C)$. Then \eqref{locsubinv} and the last sentence in Theorem~1.6 of \cite{Seneta} imply that $m^\phi|_C=0$. We can therefore apply \eqref{estphipv} to any $\mu\in CE^*v$, and deduce that $\phi(p_v)=0$.

Now suppose that $\beta=\ln\rho(A_C)$. Then by hypothesis $v\notin C$, and there exists $\lambda\in CE^*v$ of the form $\lambda=e\mu$, where $e\in E^1$, $r(e)\in C$ and $s(e)\notin C$. If $m^\phi|_C=0$, then we can apply \eqref{estphipv} to $\mu$ and deduce that $\phi(p_v)=0$. So we suppose  that $m^\phi|_C\not=0$. Then \eqref{locsubinv} and Theorem~1.6 of \cite{Seneta} imply that $m^\phi|_C$ is a multiple of the Perron-Frobenius eigenvector for $A_C$. Since
\begin{align}\label{locsubinv2}
(A_C(m^\phi|_C))_{r(e)}&\leq (A_C(m^\phi|_C))_{r(e)}+A(r(e),s(e))m^\phi_{s(e)}\\&\leq((Am^\phi)|_C)_{r(e)}\leq e^\beta m^\phi_{r(e)}=\rho(A_C)(m^\phi|_C)_{r(e)},\notag
\end{align}
and the left and right ends of \eqref{locsubinv2} are equal, we deduce that $A(r(e),s(e))m^\phi_{s(e)}=0$ and $\phi(p_{r(\mu)})=m^\phi_{s(e)}=0$; now \eqref{estphipv} implies that $\phi(p_v)=0$.
\end{proof}

\begin{proof}[Proof of Proposition~\ref{old2.5}]
Since $(E\backslash H)^*=\{\mu\in E^*:s(\mu)\notin H\}$, we can deduce from \cite[Proposition~2.1(a)]{aHLRS1} that $\psi\circ q_H$ is a KMS state if and only if $\psi$ is. Since $q_H$ is surjective, $q_H^*$ is injective, and it is clearly weak* continuous and affine. To see the assertion about surjectivity, suppose that $\{C\in H/\!\!\sim:\ln\rho(A_C)>\beta\}$ generates $H$ and $\phi$ is a KMS$_\beta$ state of $(\T C^*(E),\alpha)$. Lemma~\ref{Lemfactor} implies that $\phi(p_v)=0$ for all $v\in H$. Now we can apply \cite[Lemma~2.2]{aHLRS1} with $\F=\{s_\mu s_\nu^*:\mu,\nu\in E^*\}$ and $P=\{p_v:v\in H\}$, and deduce that $\phi$ factors through a state of $\T C^*(E)/J_H=\T C^*(E)/\ker q_H$. Thus if $H\not= E^0$, there is a state $\psi$ of $\T C^*(E\backslash H)$ such that $\phi=\psi\circ q_H$. Since $q_H$ is surjective and is equivariant for the various actions $\alpha$, $\psi$ is a KMS$_\beta$ state of $(\T C^*(E\backslash H),\alpha)$.
\end{proof}

The analogue of Proposition~\ref{old2.5} for the graph algebra $C^*(E)$ has a slightly different hypothesis: it suffices that $\{C:\ln\rho(A_C)>\beta\}$ generates $H$ as a \emph{saturated} hereditary set. This happens because the identification of $C^*(E)/I_H$ with $C^*(E\backslash H)$ only works when $H$ is saturated (compare Proposition~\ref{quotmapH} with \cite[Theorem~4.1]{BPRS} or \cite[Theorem~4.9]{CBMS}).

\begin{prop}\label{old2.5c}
Suppose that $H$ is a saturated hereditary subset of $E^0$, and write $\bar q_H$ for the canonical surjection of $C^*(E)$ onto $C^*(E\backslash H)$. Then for every $\beta\in [0,\infty)$, $\bar q_H^*:\psi\mapsto \psi\circ \bar q_H$ is an affine injection of the KMS$_\beta$ simplex of $(C^*(E\backslash H),\alpha)$ into the KMS$_\beta$ simplex of $(C^*(E),\alpha)$. If $\{C\in H/\!\!\sim:\ln\rho(A_C)>\beta\}$ generates $H$ as a saturated hereditary subset of $E^0$, then $\phi(p_v)=0$ for every KMS$_\beta$ state $\phi$ on $C^*(E)$ and every $v\in H$; if in addition $H$ is not all of $E^0$, then $\bar q_H^*$ is surjective.
\end{prop}

For the proof we need a simple lemma. Recall from the proof of \cite[Corollary~6.1]{aHLRS1}, for example,
that the saturation $\Sigma H$ of a hereditary set $H$ can be viewed as 
$\bigcup_{k=0}^\infty S_kH$, where $S_kH$ are the subsets of $E^0$ defined recursively 
by 
\begin{equation}\label{eq:satdef}
S_0H=H\quad\text{and}\quad S_{k+1}H=S_kH\cup\{v:s(vE^1)\subset S_kH\}.
\end{equation}

\begin{lem}\label{extend=}
Suppose that $H$ is a hereditary subset of $E^0$, $\beta\in [0,\infty)$, and $\phi,\psi$ are KMS$_\beta$ states on $(C^*(E), \alpha)$. 
\begin{enumerate}
\item\label{extenda} If $\phi(p_v)=\psi(p_v)$ for all $v\in H$, then $\phi=\psi$ on the ideal $I_H$ of $C^*(E)$ generated by $\{p_v:v\in H\}$.
\item\label{extendb} If $\phi(p_v)=0$ for all $v\in H$, then $\phi(p_v)=0$ for all $v$ in the saturation $\Sigma H$.
\end{enumerate}
\end{lem}

\begin{proof}
For \eqref{extenda}, we first claim that $\phi(p_v)=\psi(p_v)$ for all $v\in \Sigma H$.  
We are given that $\phi(p_v)=\psi(p_v)$ for $v\in S_0H$. Suppose that 
$\phi(p_v)=\psi(p_v)$ for $v\in S_kH$. Then for $v\in S_{k+1}H$ and $e\in vE^1$, we have 
$s(e)\in S_kH$, and
\begin{align}\label{extendcomp}
\phi(p_v)
&= \phi\Big(\sum_{e \in vE^1} s_e s^*_e\Big)
= \sum_{e \in vE^1} e^{-\beta} \phi(p_{s(e)})\\
&=\sum_{e \in vE^1} e^{-\beta} \psi(p_{s(e)})=\psi(p_v).\notag
\end{align}
Thus by induction we have $\phi(p_v)=\psi(p_v)$ for all $v\in S_kH$ and all $k$, and hence for all $v\in \Sigma H$, as claimed.

Next, we recall that
\[
I_H=\clsp\{s_\mu s_\nu^*:s(\mu)=s(\nu)\in \Sigma H\}
\]
(see \cite[Lemma~4.3]{BPRS}, for example). For a typical spanning element $s_\mu s_\nu^*$, Equation (2.1) in \cite{aHLRS1} says that
\[
\phi(s_\mu s_\nu^*)=\delta_{\mu,\nu}e^{-\beta |\mu|}\phi(p_{s(\mu)})=\delta_{\mu,\nu}e^{-\beta |\mu|}\psi(p_{s(\mu)})=\psi(s_\mu s_\nu^*),
\]
and it follows from linearity and continuity that $\phi=\psi$ on $I_H$.

For \eqref{extendb}, we repeat the induction argument of the first paragraph, and in particular the computation in the first line of \eqref{extendcomp}.
\end{proof}

\begin{proof}[Proof of Proposition~\ref{old2.5c}]
As in the proof of Proposition~\ref{old2.5}, $\bar q_H^*:\psi\mapsto \psi\circ \bar q_H$ is an affine injection of the KMS$_\beta$ simplex of $C^*(E\backslash H)$ into the KMS$_\beta$ simplex of $C^*(E)$. Suppose that $\phi$ is a KMS$_\beta$ state of $C^*(E)$. Then applying Proposition~\ref{old2.5} to the hereditary closure $H_0$ of $\{C\in H/\!\!\sim:\beta<\ln\rho(A_C)\}$ shows that $\phi(\bar p_v)=\phi\circ \pi_E(p_v)=0$ for $v\in H_0$. Thus $\{v \in E^0 : \psi(\bar p_v) = 0\}$ contains $H_0$, and hence by Lemma~\ref{extend=} contains $\Sigma H_0=H$. Now \cite[Lemma~2.2]{aHLRS1} implies that $\phi$ factors through a state of $C^*(E)/I_H$, and hence there is a state $\psi$ of $C^*(E\backslash H)$ such that $\phi=\psi\circ \bar q_H$. Then the surjectivity of $\bar q_H$ implies that $\psi$ is a KMS$_\beta$ state, and $\bar q_H^*(\psi)=\phi$.
\end{proof}

\section{KMS states on Toeplitz algebras}\label{sec:ToeplitzKMS}

We suppose that $E$ has at least one cycle, so that $\rho(A)\geq 1$ (by Lemma~A.1 in \cite{aHLRS1}), and the critical inverse temperature $\ln \rho(A)\geq 0$. Since a Seneta decomposition of $A$ is upper triangular as a block matrix, we have
\[
\rho(A)=\max\{\rho(A_C):C\in E^0/\!\!\sim\text{ is a nontrivial strongly connected component}\}.
\]
We therefore focus on the set 
\begin{equation}\label{rightevalue}
\{C\in E^0/\!\!\sim \;:\rho(A_C)=\rho(A)\}.
\end{equation}
of \emph{critical components} of $E$, and in particular on the set $\mc=\mc(E)$ of \emph{minimal critical components} that are minimal in the induced partial order on the set \eqref{rightevalue}. 

The results of the previous section imply that if $\beta<\ln\rho(A)$, then every KMS$_\beta$ state on $\T C^*(E)$ vanishes on the hereditary closure of $\{C:\ln\rho(A_C)=\ln\rho(A)\}$. This hereditary closure is the same as that of $\mc(E)$. So the location of the minimal critical components in the graph plays an important role in our analysis. Because the minimal critical components are minimal in \eqref{rightevalue}, they cannot talk to each other. Thus in a Seneta decomposition of the vertex matrix $A$, our conventions ensure that the diagonal blocks $\{A_C:C\in \mc(E)\}$ associated to the minimal critical components appear in the decomposition above other critical components $A_D$. 

The next result is a new version of \cite[Theorem~2.1(a)]{aHLRS1}.

\begin{prop}\label{discardbottom}
Suppose $E$ has at least one cycle. Let $K=\bigcup_{C\in\mc(E)} C$, let $H$ be the hereditary closure of $K$, and let $L$ be the union of the nontrivial strongly connected components. Let $\beta\in \RR$. Then
\begin{enumerate}
\item\label{elempropa} $\rho(A_{E^0\backslash H})<\rho(A)$;
\item\label{elempropb} if $\phi$ is a KMS$_{\ln\rho(A)}$ state of $(\T C^*(E),\alpha)$, then $\phi(p_v)=0$ for $v\in H\backslash K$;
\item\label{elempropc} if $E^0$ is the hereditary closure of $K$ and $\phi$ is a KMS$_\beta$ state of $(\T C^*(E),\alpha)$, then $\ln\rho(A)\leq \beta$;
\item\label{elempropd} if $E^0$ is the hereditary closure of $L$ and $\phi$ is a KMS$_\beta$ state of $(\T C^*(E),\alpha)$, then there is a nontrivial component $C$ with $\ln\rho(A_C)\leq \beta$;
\item\label{elemprope} if $E^0$ is the saturated hereditary closure of $L$ and $\phi$ is a KMS$_\beta$ state of $(C^*(E),\alpha)$, then there is a nontrivial component $C$ with $\ln\rho(A_C)\leq \beta$.
\end{enumerate}
\end{prop}

\begin{proof}
Since every minimal element of \eqref{rightevalue} is contained in $H$, so is every other strongly connected component $C$ in \eqref{rightevalue}. Thus $\rho(A_C)<\rho(A)$ for every strongly connected component $C$ that is contained in $E^0\backslash H$, and 
\[
\rho(A_{E^0\backslash H})= \max\{\rho(A_C):C\in E^0/\!\!\sim,\ C\subset E^0\backslash H\}<\rho(A),
\]
which is \eqref{elempropa}.  Next suppose that $\phi$ is a KMS$_{\ln\rho(A)}$ state on $(\T C^*(E),\alpha)$. We set things up so $\mc(E)=\{C\in \mc(E):\rho(A_C)=\rho(A)\}$, so we can apply Lemma~\ref{Lemfactor} with $\beta=\ln\rho(A)$ and $\Cc_1=\mc(E)$, and \eqref{elempropb} follows. 

For \eqref{elempropc}, we suppose that $\ln\rho(A)>\beta$. Then $\mc(E)\subset\{C\in E^0/\!\!\sim: \ln\rho(A_C)>\beta\}$, and hence the hypothesis implies that $\{C: \ln\rho(A_C)>\beta\}$ generates $E^0$. So Proposition~\ref{old2.5} applies with $H=E^0$. Thus $\phi(p_v)=0$ for all $v\in E^0$, and $1=\phi(1)=\sum_{v\in E^0}\phi(p_v)=0$, which is a contradiction. A similar argument gives \eqref{elempropd}. For \eqref{elemprope}, we repeat the argument yet again, using Proposition~\ref{old2.5c} instead of Proposition~\ref{old2.5}.
\end{proof}

\begin{rem}
If the hereditary closure $G$ of $L$ is not all of $E^0$, then $\rho(A\backslash G)=0$, and Theorem~3.1 of \cite{aHLRS1} applies to $E\backslash G$ and every $\beta\in \RR$. Thus if $\beta<\ln\rho(A_C)$ for every nontrivial component $C$, there is a $(|E^0\backslash G|-1)$-dimensional simplex of KMS$_\beta$ states on $\T C^*(E\backslash G)$. It follows from Proposition~\ref{old2.5} that there is also a $(|E^0\backslash G|-1)$-dimensional simplex of KMS$_\beta$ states on $\T C^*(E)$. Whether any of these factor through $C^*(E)$ will depend on whether $E\backslash \Sigma G$ has sources (see \cite[Corollary~6.1]{aHLRS1}), and Example~\ref{EminusGsourced} shows that $E\backslash \Sigma G$ can have sources.
\end{rem}

Proposition~\ref{discardbottom} implies that the KMS$_{\ln \rho(A)}$ simplex does not see the set $H\backslash K$, and hence (via \cite[Lemma~2.2]{aHLRS1}) that the KMS$_{\ln \rho(A)}$ states vanish on the ideal $J_{H\backslash K}$ generated by $\{p_v:v\in H\backslash K\}$.  Our next result describes how the minimal critical components give rise to KMS$_{\ln \rho(A)}$ states.

\begin{thm}\label{KMScrit}
Suppose that $E$ is a directed graph with at least one cycle. Let $K=\bigcup_{C\in \mc(E)}C$, and let $H:=\{v\in E^0:KE^*v\not=\emptyset\}$ be the hereditary closure of $K$. 
\begin{enumerate}
\item\label{crita} Let $C\in \mc(E)$ be a minimal critical component, and let $x^C$ be the unimodular Perron-Frobenius eigenvector of $A_C$ (that is, the one with $\|x^C\|_1=1$). Define a vector $z^C\in [0,\infty)^{E^0\backslash H}$ by 
\begin{equation}\label{defz}
z^C:=\rho(A)^{-1}(1-\rho(A)^{-1}A_{E^0\backslash H})^{-1}A_{E^0\backslash H,\,C}x^C.
\end{equation}
Then there is a KMS$_{\ln\rho(A)}$ state $\psi_C$ of $(\T C^*(E),\alpha)$ such that
\begin{equation}\label{formphiC}
\psi_C(s_\mu s_\nu^*)=
\delta_{\mu,\nu}\rho(A)^{-|\mu|}(1+\|z^C\|_1)^{-1}\begin{cases}
z^C_{s(\mu)}&\text{if $s(\mu)\in E^0\backslash H$}\\
x^C_{s(\mu)}&\text{if $s(\mu)\in C$}\\
0&\text{if $s(\mu)\in H\backslash C$.}
\end{cases}
\end{equation}
The state $\psi_C$ factors through a KMS$_{\ln\rho(A)}$ state $\bar\psi_C$ of $(C^*(E),\alpha)$.

\item\label{critb} The map $t\mapsto \sum_{C\in \mc(E)}t_C\psi_C$ is an affine isomorphism of 
\[
S_E:=\Big\{t\in[0,1]^{\mc(E)}:\sum_{C\in \mc(E)} t_C=1\Big\} 
\]
onto a simplex $\Sigma_{\mc(E)}$ of KMS$_{\ln\rho(A)}$ states of $(\T C^*(E),\alpha)$. Every KMS$_{\ln\rho(A)}$ state of $(\T C^*(E),\alpha)$ is a convex combination of a state of the form $q_H^*(\phi)=\phi\circ q_H$ and a state in $\Sigma_{\mc(E)}$.
\end{enumerate}
\end{thm}

The idea in part~\eqref{crita} is that the values of a KMS state on vertices in $C$
contribute to the values $\phi(p_v)$ for $v\in E^0\backslash H$ when there are paths 
$\lambda$ from $C$ to $v$. As 
discussed at the beginning of Section~3 of \cite{aHLRS1}, for $\beta > \ln\rho(A_{E^0 \backslash H})$ the series $\sum^\infty_{n=0} e^{-\beta n} A_{E^0 \backslash H}^n$ converges  in operator norm to $(1 - e^{-\beta} A_{E^0 \backslash H})^{-1}$, and so
\begin{equation}\label{eq:inverse series}
(1 - e^{-\beta} A_{E^0 \backslash H})^{-1}(v,w)
= \sum^\infty_{n=0} e^{-\beta n} A_{E^0 \backslash H}^n(v,w)
= \sum_{\lambda \in v E^* w} e^{-\beta |\lambda|}.
\end{equation}
Since $\rho(A) > \rho(A_{E^0 
\backslash H})$,  the $(E^0\backslash H)\times C$ matrix 
$(1-\rho(A)^{-1}A_{E^0\backslash H})^{-1}A_{E^0\backslash H,\,C}$ in~\eqref{defz} has entries
\begin{equation}\label{eq:zCseries}
\big((1-\rho(A)^{-1}A_{E^0\backslash H})^{-1}A_{E^0\backslash H,\,C}\big)(v,w)
	= \sum_{e \in (E^0 \backslash H) E^1 w}\; \sum_{\mu \in v E^* r(e)} \rho(A)^{-|\mu|}.
\end{equation}
We use \eqref{eq:inverse series} in the proof of part~(\ref{crita}) 
and~\eqref{eq:zCseries} in the proof of part~(\ref{critb}), and  again in Lemma~\ref{factorthruC*} and Theorem~\ref{thm:altogether}(\ref{it:critstates}).

\begin{proof}[Proof of Theorem~\ref{KMScrit}\,\eqref{crita}]
We partition $E^0$ as $(E^0\backslash H)\cup C\cup (H\backslash C)$, and claim that the vector $(z^C,x^C,0)$ satisfies
\begin{equation}\label{eigenvectorforA}
A(z^C,x^C,0)=\rho(A)(z^C,x^C,0).
\end{equation}
Since $C$ is minimal, it does not talk to any of the other components in $H\backslash C$, and we have
\begin{equation}\label{decompA}
A(z^C,x^C,0)=(A_{E^0\backslash H}z^C+A_{E^0\backslash H,C}x^C,A_Cx^C,0).
\end{equation}
We know that $A_Cx^C=\rho(A)x^C$, so we concentrate on the first term. 
Proposition~\ref{discardbottom} implies that $\rho(A_{E^0\backslash H})<\rho(A)$. 
Since $e^{-\rho(A)}=\rho(A)^{-1}$, ~\eqref{eq:inverse series} gives
\[
z^C=\sum_{n=0}^\infty\rho(A)^{-n-1}A_{E^0\backslash H}^nA_{E^0\backslash H,\,C}x^C,
\] 
and we have
\begin{align*}
A_{E^0\backslash H}z^C+A_{E^0\backslash H,C}x^C
&=A_{E^0\backslash H}\Big(\sum_{n=0}^\infty \rho(A)^{-n-1}A_{E^0\backslash H}^nA_{E^0\backslash H,C}x^C\Big)+A_{E^0\backslash H,C}x^C\\
&=\Big(\sum_{m=1}^\infty \rho(A)^{-m}A_{E^0\backslash H}^mA_{E^0\backslash H,C}x^C\Big)+A_{E^0\backslash H,C}x^C\\
&=\sum_{m=0}^\infty \rho(A)^{-m}A_{E^0\backslash H}^mA_{E^0\backslash H,C}x^C\\
&=\rho(A)z^C.
\end{align*}
From this and \eqref{decompA}, we deduce that $(z^C,x^C,0)$ satisfies \eqref{eigenvectorforA}, as claimed.

Since $x^C$ is unimodular, $m:=(1+\|z^C\|_1)^{-1}(z^C,x^C,0)$ satisfies $\|m\|_1=1$, and hence is a probability measure on $E^0$. Equation~\eqref{eigenvectorforA} implies that $Am=\rho(A)m$. Thus Proposition~4.1 of \cite{aHLRS1} implies that there is a KMS$_{\ln\rho(A)}$ state $\psi_C$ on $(\T C^*(E),\alpha)$ satisfying \eqref{formphiC},
and that $\psi_C$ factors through a KMS$_{\ln\rho(A)}$ state of $(C^*(E),\alpha)$.
\end{proof}

The double sum appearing on the right-hand side of~\eqref{eq:zCseries} is parametrised by paths in $vE^*w$ 
of the form $\mu e$, where $r(e)$ is in $E^0\backslash C$ and $\mu$ is a path in 
$E\backslash H$. We say that such paths \emph{make a quick exit from $C$}. For a 
minimal critical component $C$, we write $\QE(C)$ for the set
\[
\QE(C) := \{\mu e : e \in E^1 C, r(e) \not\in C, \mu \in E^* r(e)\}
\]
of paths which start in $C$ and make a quick exit from $C$, and $\QE(K):=\bigcup_{C\in 
\mc(E)} \QE(C)$. With this notation, the right-hand side of~\eqref{eq:zCseries} becomes 
\[
\sum_{\lambda \in v \QE(C)w} \rho(A)^{-(|\lambda|-1)}.
\]

\begin{lem}\label{lemonQE}
The projections $\{s_\lambda s_\lambda^*:\lambda\in \QE(K)\}$ are mutually orthogonal.
\end{lem}

\begin{proof}
Suppose that $\mu,\nu\in \QE(K)$
and $\mu\not=\nu$. If $|\mu|=|\nu|$, then $(s_\mu s_\mu^*)(s_\nu s_\nu^*)=s_\mu(s_\mu^*s_\nu)s_\nu^*=0$. So suppose that one path is longer, say $|\mu|>|\nu|$. Then $s(\nu)$ is in $K$ and $s(\mu_{|\nu|})$ is not in $K$ because the different minimal critical components do not talk to each other. Thus $\mu$ does not have the form $\nu\mu'$, and we have $s_\mu^*s_\nu=0$, which implies that $(s_\mu s_\mu^*)(s_\nu s_\nu^*)=0$.
\end{proof}

\begin{proof}[Proof of Theorem~\ref{KMScrit}\,\eqref{critb}]
Suppose that $\phi$ is a KMS$_{\ln\rho(A)}$ state of $(\T C^*(E),\alpha)$, and consider $m^\phi=(\phi(p_v))$, which by \cite[Proposition~2.1(c)]{aHLRS1} satisfies the subinvariance relation $Am^\phi\leq \rho(A)m^\phi$. Suppose that $C\in \mc=\mc(E)$. Proposition~\ref{discardbottom} implies that $m^\phi_v=0$ for $v\in H\backslash K$, which since the minimal critical components do not talk to each other implies that $(Am^\phi)|_C=A_C(m^\phi|_C)$. So subinvariance implies that 
\[
A_C(m^\phi|_C)=(Am^\phi)|_C\leq \rho(A)m^\phi|_C=\rho(A_C)m^\phi|_C;
\]
now \cite[Theorem~1.6]{Seneta} implies that we have equality throughout, and that $m^\phi|_C$ is a multiple of the unimodular Perron-Frobenius eigenvector $x^C$ for $A_C$. We define $t_C\in [0,\infty)$ by $m^\phi|_C=t_C(1+\|z^C\|_1)^{-1}x^C$. 

We claim that $\sum_{C\in \mc}t_C\leq 1$. For $v\in E^0\backslash H$,  
Lemma~\ref{lemonQE} implies that $\phi(p_v)\geq \sum_{\lambda\in v\QE(K)}\phi(s_\lambda 
s_\lambda^*)$. Now we calculate, using \cite[Proposition~2.1(a)]{aHLRS1} 
and~\eqref{eq:zCseries}: 
\begin{align}\label{lbforphip}
\phi(p_v)&\geq \sum_{\lambda\in v\QE(K)}\phi(s_\lambda s_\lambda^*)=\sum_{\lambda\in v\QE(K)}\rho(A)^{-|\lambda|}\phi(p_{s(\lambda)})\\
&=\sum_{C\in \mc}t_C(1+\|z^C\|_1)^{-1}\Big(\sum_{\lambda\in v\QE(C)}\rho(A)^{-|\lambda|}x^C_{s(\lambda)}\Big)\notag\\
&=\sum_{C\in \mc}t_C(1+\|z^C\|_1)^{-1}\Big(\sum_{w\in C}\;\sum_{\lambda\in v\QE(C)w}\rho(A)^{-|\lambda|}x^C_w\Big)\notag\\
&=\sum_{C\in \mc}t_C(1+\|z^C\|_1)^{-1}\Big(\sum_{w\in 
C}\rho(A)^{-1}\big((1-\rho(A)^{-1}A_{E^0\backslash H})^{-1}A_{E^0\backslash 
H,\,C})\big)(v,w)x^C_w\Big)\notag\\
&=\sum_{C\in \mc}t_C(1+\|z^C\|_1)^{-1}z^C_v.\notag
\end{align}
For $v\in C$ we have $\phi(p_v)=t_C(1+\|z^C\|_1)^{-1}x^C_v$ by definition of $t_C$. Thus
\begin{align}\label{subinvofdiff}
1=\phi(1)&=\sum_{v\in E^0}\phi(p_v)\geq\sum_{v\in E^0\backslash H}\phi(p_v)+\sum_{C\in \mc}\sum_{v\in C}\phi(p_v)\\
&\geq \sum_{v\in E^0\backslash H}\sum_{C\in \mc}t_C(1+\|z^C\|_1)^{-1}z^C_v+\sum_{C\in \mc}\sum_{v\in C}t_C(1+\|z^C\|_1)^{-1}x^C_v\notag\\
&=\sum_{C\in \mc}t_C(1+\|z^C\|_1)^{-1}\|z^C\|_1+\sum_{C\in \mc}t_C(1+\|z^C\|_1)^{-1}\notag\\
&=\sum_{C\in \mc}t_C,\notag
\end{align} 
as claimed.

The states $\psi_C$ in part~\eqref{crita} are KMS$_{\ln\rho(A)}$ states with $m^{\psi_C}=(1+\|z^C\|_1)^{-1}(z^C,x^C,0)$, and hence \eqref{eigenvectorforA} says that $Am^{\psi_C}=\rho(A)m^{\psi_C}$. We know from \cite[Proposition~2.1(c)]{aHLRS1} that $m^\phi$ is a probability measure with $Am^\phi\leq \rho(A)m^\phi$. Thus 
\begin{align}
A\Big(m^\phi-&\sum_{C\in\mc}t_Cm^{\psi_C}\Big)=Am^\phi-\sum_{C\in \mc} t_CAm^{\psi_C}
=Am^\phi-\sum_{C\in \mc} t_C\rho(A)m^{\psi_C}\label{msubinv}\\
&\leq \rho(A)m^\phi-\sum_{C\in \mc} t_C\rho(A)m^{\psi_C}
=\rho(A)\Big(m^\phi-\sum_{C\in\mc}t_Cm^{\psi_C}\Big).\notag
\end{align}
For $v\in E^0\backslash H$, we saw in \eqref{lbforphip} that
\[
m^\phi_v=\phi(p_v)\geq\sum_{C\in \mc}t_C(1+\|z^C\|_1)^{-1}z^C_v=\sum_{C\in \mc}t_Cm^{\psi_C}_v.
\]
For $v\in K$, say $v\in C$, we have from the definition of $t_C$ that
\[
m^\phi_v=t_C(1+\|z^C\|_1)^{-1}x^C_v=t_C\psi_C(p_v)=t_Cm^{\psi_C}_v;
\]
for $v\in H\backslash K$, Proposition~\ref{discardbottom} gives $m^\phi_v=m^{\psi_C}_v=0$ for all $C$. Thus the difference satisfies $\big(m^\phi-\sum_{C\in \mc} t_Cm^{\psi_C}\big)\big|_H=0$.

If $\sum_{C\in \mc}t_C=1$, then we have $m^\phi=\sum_C t_Cm^{\psi_C}$ because both are probability measures and $m^\phi\geq\sum_C t_Cm^{\psi_C}$. Then, since $\phi$ and $\sum_Ct_C\psi_C$ are KMS$_{\ln\rho(A)}$ states of $(\T C^*(E),\alpha)$ which agree on projections, Proposition~2.1 of \cite{aHLRS1} implies that $\phi=\sum_{C} t_C\psi_C$.

If $\sum_{C\in \mc}t_C<1$, then
\[
m:=\Big(1-\sum_{C\in\mc}t_C\Big)^{-1}\Big(m^\phi-\sum_{C\in\mc}t_Cm^{\psi_C}\Big)\Big|_{E^0\backslash H}
\]
is a probability measure, and the calculation \eqref{msubinv} implies that $m$ is subinvariant for the graph $E\backslash H$. Since $\rho(A_{E^0\backslash H})<\rho(A)$, applying \cite[Theorem~3.1]{aHLRS1} to the graph $E\backslash H$, with $\beta=\ln\rho(A)$ and $\epsilon=(1-\rho(A)^{-1}A_{E^0\backslash H})^{-1}m$, gives a KMS$_{\ln\rho(A)}$ state $\phi_\epsilon$ on $(\T C^*(E\backslash H),\alpha)$ such that $\phi_\epsilon(p_v)=m_v$ for $v\in E^0\backslash H$. Now $(1-\sum_C t_C)(\phi_\epsilon\circ q_H)+\sum_Ct_C\psi_C$ is a KMS$_{\ln\rho(A)}$ state on $(\T C^*(E),\alpha)$ which agrees with $\phi$ on vertex projections, and hence
\[
\phi=\Big(1-\sum_{C\in\mc} t_C\Big)(\phi_\epsilon\circ q_H)+\sum_Ct_C\psi_C.\qedhere
\]
\end{proof}

Since $\rho(A_{E^0\backslash H})<\rho(A)$, Theorem~3.1 of \cite{aHLRS1} describes the 
KMS$_{\ln\rho(A)}$ states on $\T C^*(E\backslash H)$. We write $y^{E\backslash H}$ for 
the vector in $[1,\infty)^{E^0\backslash H}$ described in \cite[Theorem~3.1(a)]{aHLRS1}, 
$\Delta^{E\backslash H}_{\ln\rho(A)}$ for the simplex $\{\epsilon:\epsilon\cdot 
y^{E\backslash H}=1\}$ in $[0,\infty)^{E^0\backslash H}$, and $\phi_\epsilon$ for the 
KMS$_{\ln\rho(A)}$ state on $\T C^*(E\backslash H)$ described in 
\cite[Theorem~3.1(b)]{aHLRS1}.

\begin{cor}\label{cor:state decomp}
Every KMS$_{\ln\rho(A)}$ state on $\T C^*(E)$ has the form
\begin{equation}\label{genKMScrit}
\phi_{r,\epsilon,t}:=r(\phi_\epsilon\circ q_H)+(1-r)\Big(\sum_{C\in \mc(E)}t_C\psi_C\Big) 
\end{equation}
for some $r\in [0,1]$, $\epsilon\in\Delta^{E\backslash H}_{\ln\rho(A)}$ and $t\in 
S_{\mc}$. We have $\phi_{r,\epsilon,t} = \phi_{r',\epsilon',t'}$ if and only if 
$(r\epsilon, (1-r)t) = (r'\epsilon', (1-r')t')$.
\end{cor}

\begin{proof}
Theorem~\ref{KMScrit}(\ref{critb}) shows that each KMS$_{\ln\rho(A)}$ state has the 
form~\eqref{genKMScrit}. 

Suppose that $(r\epsilon, (1-r)t) = (r'\epsilon', (1-r')t')$. Then $(1 - r)\sum t_C 
\psi_C = (1 - r')\sum t'_C \psi_C$, and so $\phi_{r, \epsilon, 
t} - \phi_{r',\epsilon',t'} = (r\phi_\epsilon - r'\phi_{\epsilon'}) \circ q_H$. 
Since $\sum t_C = \sum t'_C = 1$, we also have $1 - r = 1 - r'$ and hence $r = r'$. 
So either $r = 0$ or $\epsilon = 
\epsilon'$, and in either case, $r\phi_\epsilon = r'\phi_{\epsilon'}$, giving 
$\phi_{r,\epsilon,t} - \phi_{r',\epsilon',t'} = 0$.

Now suppose that $\phi_{r, \epsilon, t} = \phi_{r', \epsilon', t'}$. Fix $C \in \mc(E)$ 
and $v \in C$. For $C' \in \mc(E)$, the formula~\eqref{formphiC} shows that 
$\psi_{C'}(p_v) = \delta_{C, C'}(1 + \|z^C\|)^{-1}x^C_v$. Since $q_H(p_v) = 0$, 
\[
0 = \phi_{r, \epsilon, t}(p_v) - \phi_{r', \epsilon', t'}(p_v)
	= \big((1-r)t_C - (1 - r')t'_C\big) (1 + \|z^C\|)^{-1} x^C_v.
\]
Parts (a)~and~(d) of \cite[Theorem~1.5]{Seneta} imply that $x^C_v > 0$, and so $(1 - 
r)t_C = (1 - r')t'_C$. It remains to show that $r \epsilon = r'\epsilon'$. 
We have $r = 1 - \|(1-r)t\|_1 = 1 - \|(1-r')t'\|_1 = r'$, and so $0	= \phi_{r, \epsilon, 
t} - \phi_{r,\epsilon',t} = r(\phi_\epsilon \circ q_H - \phi_{\epsilon'}\circ q_H)$.
If $r = 0$, then we trivially have $r\epsilon = r' \epsilon'$. Suppose that $r \not= 0$. 
Then $\phi_\epsilon \circ q_H = \phi_{\epsilon'} \circ q_H$. Proposition~\ref{old2.5} 
implies that $q_H^*$ is injective, so $\phi_\epsilon = \phi_{\epsilon'}$; since  $\epsilon \mapsto \phi_\epsilon$ is 
injective  \cite[Theorem~3.1(b)]{aHLRS1}, we deduce that $\epsilon = \epsilon'$.
\end{proof}

\section{The KMS simplices for a fixed inverse temperature}\label{sec:allKMS}

In this section, we consider a finite directed graph $E$ and a real number $\beta$, and aim to describe the extreme points of the KMS$_\beta$ simplices of 
$\T C^*(E)$ and $C^*(E)$. The states described in Theorem~\ref{KMScrit} will be some of them. We generate some more candidates by applying \cite[Theorem~3.1]{aHLRS1} to a graph of the form $E\backslash H$. We continue to use the recursive description of the saturation $\Sigma H$ described on page~\pageref{eq:satdef}.

\begin{prop}\label{defphiv}
Suppose that $H$ is a hereditary subset of $E^0$ and $\beta>\ln\rho(A_{E^0\backslash H})$. For each $v\in E^0\backslash H$ the series $\sum_{\mu\in (E\backslash H)^*v}e^{-\beta|\mu|}$ converges with sum $y_v\geq 1$; let $y$ be the the vector $(y_v)$ in $[1,\infty)^{E^0\backslash H}$. Then for each $v\in E^0\backslash H$, there is a KMS$_\beta$ state $\phi^H_v$ of $\T C^*(E\backslash H)$ such that
\begin{equation}\label{eq:phiv formula}
\phi_v^H(s_\mu s^*_\nu) = \delta_{\mu,\nu} e^{-\beta|\mu|} (1-e^{-\beta}A_{E^0\backslash H})^{-1}(s(\mu),v)y_v^{-1}\quad\text{ for $\mu,\nu \in (E\backslash H)^*$.}
\end{equation}
The states $\{\phi_v^H:v\in E^0\backslash H\}$ are the extremal KMS$_\beta$ states of $\T C^*(E\backslash H)$. 
\end{prop}

\begin{proof}
Applying \cite[Theorem~3.1(a)]{aHLRS1} to $E\backslash H$ shows that the series defining $y_v$ converges. We define $\epsilon^v\in [0,\infty)^{E^0\backslash H}$ by $\epsilon^v_u=\delta_{u,v}y_v^{-1}$. Then $\epsilon^v\cdot y=1$, and the corresponding probability measure $m^v=(1-e^{-\beta}A_{E^0\backslash H})^{-1}\epsilon^v$ in \cite[Theorem~3.1(a)]{aHLRS1} has entries
\[
m^v_w=(1-e^{-\beta}A_{E^0\backslash H})^{-1}(w,v)y_v^{-1}\quad\text{for $w\in E^0\backslash H$.}
\] 
Thus by \cite[Theorem~3.1(b)]{aHLRS1}, there is a KMS$_\beta$ state $\phi_v^H$ of $\T C^*(E\backslash H)$ satisfying \eqref{eq:phiv formula}. It follows from \cite[Theorem~3.1(c)]{aHLRS1} that the $\phi_v^H$ are the extreme points of the simplex of KMS$_\beta$ states (as observed in \cite[Remark~3.2]{aHLRS1}).
\end{proof}

\begin{cor}\label{defphiv2}
Let $v\in E^0$ and $\beta>0$. Suppose that there is a hereditary subset $H$ of $E^0$ such that $v\notin H$ and $\ln\rho(A_{E^0\backslash H})<\beta$. Then there is a KMS$_\beta$ state $\phi_{\beta,v}$ of $(\T C^*(E),\alpha)$ such that for every pair $\mu,\nu\in E^*$, we have
\begin{equation}\label{defphibetav}
\phi_{\beta,v}(s_\mu s^*_\nu)=
\begin{cases}
0&\text{if $s(\mu)E^*v=\emptyset$}\\
\delta_{\mu,\nu} \Big(e^{-\beta|\mu|}\sum_{\lambda\in s(\mu)E^*v}e^{-\beta|\lambda|}\Big)y_v^{-1}&\text{if $s(\mu)E^*v\not=\emptyset$;}
\end{cases}
\end{equation}
for every $H$ satisfying these hypotheses, we have $\phi_{\beta,v}=\phi^H_v\circ q_H$.
\end{cor}

Notice that \eqref{defphibetav} implies that the state $\phi_{\beta,v}$ does not depend on the choice of the hereditary set $H$ satisfying $v\notin H$ and $\ln\rho(A_{E^0\backslash H})<\beta$.

\begin{proof}
Proposition~\ref{defphiv} gives us a KMS$_\beta$ state $\phi^H_v$ of $(\T C^*(E\backslash H),\alpha)$. Because $H$ is hereditary, every path $\lambda$ in $E^*v$ lies entirely in $E\backslash H$. Thus \eqref{eq:phiv formula} implies that for every $\mu,\nu\in (E\backslash H)^*$, we have
\[
\phi_v^H(s_\mu s^*_\nu) = \delta_{\mu,\nu}\Big(e^{-\beta|\mu|}\sum_{\lambda\in s(\mu)(E\backslash H)^*v}e^{-\beta|\lambda|}\Big)y_v^{-1}=\delta_{\mu,\nu} \Big(e^{-\beta|\mu|}\sum_{\lambda\in s(\mu)E^*v}e^{-\beta|\lambda|}\Big)y_v^{-1};
\]
notice that \eqref{eq:phiv formula} is zero if $s(\mu)E^*v=\emptyset$, and in that case we need to interpret the empty sum on the right-hand side as $0$. For  $\mu,\nu\in E^*$ with $s(\mu)=s(\nu)\in H$, we have $q_H(s_\mu s^*_\nu)=0$. Thus for arbitrary $\mu,\nu\in E^*$ with $s(\mu)=s(\nu)$, we have
\[
\phi_v^H\circ q_H(s_\mu s^*_\nu)=
\begin{cases}
0&\text{if $s(\mu)=s(\nu)\in H$}\\
\delta_{\mu,\nu} \Big(e^{-\beta|\mu|}\sum_{\lambda\in s(\mu)E^*v}e^{-\beta|\lambda|}\Big)y_v^{-1}&\text{if $s(\mu)=s(\nu)\in E^0\backslash H$,}
\end{cases}
\]
and $\phi_{\beta,v}:= \phi^H_v\circ q_H$ is a KMS$_\beta$ state of $(\T C^*(E),\alpha)$ satisfying \eqref{defphibetav}.
\end{proof}

\begin{thm}\label{thm:altogether}
Suppose that $E$ is a finite directed graph and $\beta$ is a real number, and denote by $\alpha$ all the actions of $\RR$ obtained by lifting gauge actions on Toeplitz algebras and graph algebras. Let $H_\beta$ be the hereditary closure in $E^0$ of $\{C \in E^0/\!\!\sim : \ln\rho(A_C)>\beta\}$. 
\begin{enumerate}
\item\label{it:nostates} If $H_\beta = E^0$, then $(\T C^*(E),\alpha)$ has no KMS$_\beta$ states. 
\item\label{it:nocrit} Suppose that $H_\beta \not= E^0$ and that $\beta > \ln\rho(A_{E^0\backslash H_\beta})$. For $v \in E^0\backslash H_\beta$, there is a KMS$_\beta$ state $\phi_{\beta,v}$ of $(\T C^*(E),\alpha)$ satisfying \eqref{defphibetav}. Then  
\[
\big\{\phi_{\beta, v}: v \in E^0 \backslash H_\beta\big\}
\] 
are the extreme points of the KMS$_\beta$ simplex of $(\T C^*(E),\alpha)$. A KMS$_\beta$ state factors through $C^*(E)$ if and only if it belongs to the convex hull of 
\[
\big\{\phi_{\beta, v}: v \text{ is a source in } E \backslash \Sigma H_\beta\big\}.
\]
\item\label{it:critstates} Suppose that $H_\beta \not= E^0$ and that $\beta = 
\ln\rho(A_{E^0 \backslash H_\beta})$. Let $K_\beta$ be the hereditary closure in $E^0$ of $\{C \in E^0/\!\!\sim : \ln\rho(A_C)\geq\beta\}$. For $v \in E^0 \backslash K_\beta$, there is a KMS$_\beta$ state $\phi_{\beta, v}$ of $(\T C^*(E),\alpha)$ satisfying \eqref{defphibetav}. For $C \in \mc(E \backslash H_\beta)$, let $\psi_C^{H_\beta}$ be the KMS$_\beta$ state of $(\T C^*(E \backslash H_\beta),\alpha)$ obtained by applying Theorem~\ref{KMScrit}(\ref{crita}) to the graph $E\backslash H_\beta$. Then the states
\begin{equation}\label{eq:extreme points}
\big\{\psi_C:=\psi_C^{H_\beta}\circ q_{H_\beta} : C \in \mc(E\backslash H_\beta)\big\} \cup
\big\{\phi_{\beta,v}: v \in E^0 \backslash{K_\beta}\big\}
\end{equation}
are the extreme points of the KMS$_\beta$ simplex of $(\T C^*(E),\alpha)$. A KMS$_\beta$ state factors through $C^*(E)$ if and only if it belongs to the convex hull of 
\begin{equation}\label{convhull}
\big\{\psi_C: C \in \mc(E\backslash H_\beta)\big\} \cup 
\big\{\phi_{\beta, v}: v \text{ is a source in } E \backslash \Sigma K_\beta\big\}.
\end{equation}
\end{enumerate}
\end{thm}

Both $H_\beta$ and $K_\beta$ are hereditary subsets of $E^0$, and $H_\beta\subset K_\beta$. Obviously the proof of the theorem must exploit the specific nature of these two sets, but some of our arguments are more general, and we separate out some lemmas. Throughout this section, $E$ is a finite directed graph.

\begin{lem}\label{lem:extremeptsfactor}
Suppose that $I$ is an ideal in a $C^*$-algebra $A$, that $\phi_1,\dots,\phi_n$ 
are states of $A$, and that $\lambda_i\in (0,\infty)$ for $1\leq i\leq n$. Then $\sum_{j=1}^n \lambda_j \phi_j$ factors through 
$A/I$ if and only if $\phi_i$ factors through $A/I$ for all $i$.
\end{lem}

\begin{proof}
If each $\phi_i$ factors through $A/I$, then so does every linear combination. So suppose that $\sum_j\lambda_j \phi_j$ factors through $C^*(E)$. For a 
positive element $a$ in $I$ and each $i$, we have 
\[
0 = \sum_{j=1}^n \lambda_j \phi_j(a) \geq \lambda_i \phi_i(a)\geq 0,
\]
and since $\lambda_i > 0$, this forces $\phi_i(a) = 0$. Since $I$ is spanned by 
its positive elements, we deduce that $\phi_i$ vanishes on $I$, and hence $\phi_i$ factors through $A/I$.
\end{proof}

\begin{lem}\label{lem:smallerC*E}
Suppose that $H \subset E^0$ is hereditary 
and that $\phi$ is a KMS$_\beta$-state of $\T C^*(E \backslash H)$ which factors 
through $C^*(E \backslash H)$. If $\phi(p_v) = 0$ for all $v \in \Sigma H \backslash H$, then the state $\phi \circ q_H$ of $\T C^*(E)$ factors through $C^*(E)$.
\end{lem}

\begin{proof}
The hypothesis says that there is a state $\bar\phi$ of $C^*(E\backslash H)$ such that  $\phi = \bar\phi \circ \pi_{E \backslash H}$.
Let $J$ be the ideal of $C^*(E \backslash H)$ generated by $\{p_v : v \in \Sigma H 
\backslash H\}$. Then \cite[Lemma~2.2]{aHLRS1} implies that $\bar{\phi}$ factors through 
$C^*(E \backslash H)/J$. Theorem~4.1(b) of \cite{BPRS} implies that there is an isomorphism of $C^*(E \backslash \Sigma H)$ onto $C^*(E \backslash H)/J$ which takes $\bar{s}_e$ to $s_e + 
J$. So 
there is a KMS$_\beta$ state $\bar{\bar{\phi}}$ of $C^*(E \backslash \Sigma H)$ such that 
$\phi = \bar{\bar{\phi}} \circ \bar{q}_{\Sigma H \backslash H} \circ \pi_{E 
\backslash H}$. By considering the images of generators of $\T C^*(E)$, one checks that 
the diagram
\begin{equation}\label{eq:CD}
\begin{tikzpicture}[>=stealth, yscale=0.6]
\node (TC*E) at (4,2) {$\T C^*(E)$};
\node (C*E) at (4,0) {$C^*(E)$};
\node (TC*E/H) at (0,2) {$\T C^*(E \backslash H)$};
\node (C*E/H) at (-4,2) {$C^*(E \backslash H)$};
\node (C*E/SH) at (-4,0) {$C^*(E \backslash \Sigma H)$}; 
\draw[->] (TC*E)--(C*E) node[right, midway] {\small$\pi_E$};
\draw[->] (C*E)--(C*E/SH) node[above, midway] {\small$\bar{q}_{\Sigma H}$};
\draw[->] (TC*E)--(TC*E/H) node[above, midway] {\small$q_H$};
\draw[->] (TC*E/H)--(C*E/H) node[above, midway] {\small$\pi_{E \backslash H}$};
\draw[->] (C*E/H)--(C*E/SH) node[left, midway] {\small$\bar{q}_{\Sigma H \backslash H}$};
\end{tikzpicture}\end{equation}
commutes. Thus $\phi \circ 
q_H$ factors through the state $\bar{\bar{\phi}} \circ \bar{q}_{\Sigma H}$ of $C^*(E)$.
\end{proof}

\begin{lem}\label{relatephis}
Suppose that $E$ is a finite directed graph with vertex matrix $A$, and that $\beta>\ln\rho(A)$. Suppose that $G$ is a hereditary subset of $E^0$, and let $y^E\in [1,\infty)^{E^0}$ and $y^{E\backslash G}$ be the vectors of \cite[Theorem~3.1]{aHLRS1} for the graphs $E$ and $E\backslash G$. If $\epsilon\in [0,1]^{E^0}$ satisfies $\epsilon\cdot y=1$ and $\epsilon|_{G}=0$, then $\epsilon|_{E\backslash G}$ satisfies $(\epsilon|_{E\backslash G})\cdot y^{E\backslash G}=1$, and the corresponding KMS$_\beta$ states on the Toeplitz algebras satisfy $\phi_\epsilon=\phi_{\epsilon|_{E\backslash G}}\circ q_G$.
\end{lem}

\begin{proof}
For $w\in E^0\backslash G$, we have $(E\backslash G)^*w=E^*w$, and hence
\[
y^{E\backslash G}_w=\sum_{\mu\in(E\backslash G)^*w}e^{-\beta|\mu|}=\sum_{\mu\in E^*w}e^{-\beta|\mu|}=y^E_w.
\]
Thus $y^{E\backslash G}=y|_{E\backslash G}$, and $1=\epsilon\cdot y^E=(\epsilon|_{E\backslash G})\cdot (y|_{E\backslash G})=(\epsilon|_{E\backslash G})\cdot y^{E\backslash G}$.

Since $G$ is hereditary, for $v\in E^0$ we have
\[
m_v=\big((1-e^{-\beta}A)^{-1}\epsilon\big)_v=\begin{cases}((1-e^{-\beta}A_{E^0\backslash G})^{-1}\epsilon|_{E\backslash G})_v&\text{if $v\in E^0\backslash G$}\\
0&\text{if $v\in G$,}
\end{cases}
\]
and hence
\[
\phi_\epsilon(p^E_v)=\begin{cases}\phi_{\epsilon|_{E\backslash G}}(p^{E\backslash G}_v)&\text{if $v\in E^0\backslash G$}\\
0&\text{if $v\in G$.}
\end{cases}
\]
Thus $\phi_\epsilon$ and $\phi_{\epsilon|_{E\backslash G}}\circ q_G$ agree on the vertex projections $\{p_v\}$ in $\T C^*(E)$, and since both are KMS$_\beta$ states, \cite[Proposition~2.1(a)]{aHLRS1} implies that they are equal.
\end{proof}

\begin{lem}\label{factorthruC*}
Suppose that $H$ is a hereditary subset of $E$ and $\beta>\ln \rho(A_{E^0\backslash H})$. Let $v\in E^0\backslash H$, and let $\phi^H_v$ be the state of $\T C^*(E\backslash H)$ described in Proposition~\ref{defphiv}. Then $\phi^H_v\circ q_H$ factors through $C^*(E)$ if and only if $v$ is a source in $E\backslash\Sigma H$.
\end{lem}

\begin{proof}
Suppose that $v$ is a source in $E\backslash\Sigma H$. Then $v$ must be a source in $E$: otherwise, we have $s(r^{-1}(v))\subset \Sigma H$, and saturation implies that $v\in \Sigma H$. In particular, $v$ is a source in $E\backslash H$, and \cite[Corollary~6.1(a)]{aHLRS1} implies that $\phi^H_v$ factors through $C^*(E\backslash H)$. With a view to applying Lemma~\ref{lem:smallerC*E}, we take $w\in \Sigma H\backslash H$. Since $\Sigma H$ is hereditary, $wE^nv=\emptyset$ for all $n$, and \eqref{eq:inverse series} implies that $(1-e^{-\beta}A_{E^0\backslash H})^{-1}(w,v)=0$. Thus 
\[
\phi^H_v(p_w)=(1-e^{-\beta}A_{E^0\backslash H})^{-1}(w,v)y_v^{-1}=0, 
\]
and Lemma~\ref{lem:smallerC*E} implies that $\phi^H_v\circ q_H$ factors through $C^*(E)$. 

Now suppose that $\phi^H_v \circ q_H$ factors through $C^*(E)$. Since 
$\phi^H_v \circ q_H(p_w)=0$ for $w \in H$, 
Lemma~\ref{extend=}(\ref{extendb}) implies that $\phi^H_v \circ q_H$ vanishes on $\{p_w : w \in \Sigma H\}$. Thus it follows from \cite[Lemma~2.2]{aHLRS1} that $\phi^H_v \circ q_H$ factors through $C^*(E \backslash \Sigma H)$. Since $\phi^H_v(p_v) \not= 0$, we deduce that $v \in E^0 \backslash \Sigma H$. Thus $\epsilon^v_w = (y^{E \backslash H}_v)^{-1} \delta_{v,w}$ vanishes for $w$ in the hereditary set $\Sigma H$, and Lemma~\ref{relatephis} implies that $\phi^H_v\circ q_H=\phi_{\epsilon^v|_{E \backslash \Sigma H}} \circ q_{\Sigma H 
\backslash H}$. Since $\phi^H_v \circ q_H$ factors through $C^*(E)$, for $w \in E^0 \backslash \Sigma H$ we have 
\begin{align*}
\phi_{\epsilon^v|_{E \backslash \Sigma H}}\Big(p_w - \sum_{e \in w(E \backslash \Sigma H)^1} s_e s^*_e\Big)
	&= \phi_{\epsilon^v|_{E \backslash \Sigma H}} \circ q_{\Sigma H \backslash H} 
		\Big(p_w - \sum_{e \in w(E \backslash H)^1} s_e s^*_e\Big) \\
	&= \phi^H_v \circ q_H\Big(p_w - \sum_{e \in wE^1} s_e s^*_e\Big) = 0. 
\end{align*}
Applying \cite[Lemma~2.2]{aHLRS1} to $E\backslash \Sigma H$ shows that 
$\phi_{\epsilon^v|_{E \backslash \Sigma H}}$ factors through $C^*(E \backslash \Sigma 
H)$. We have $\beta > \rho(A_{E^0 \backslash H})$, so Corollary~6.1(a) of \cite{aHLRS1} 
implies that $\epsilon^v|_{E \backslash \Sigma H}$ is supported on the sources of $E 
\backslash \Sigma H$, and hence $v$ is a source in $E \backslash \Sigma H$.
\end{proof}

\begin{proof}[Proof of Theorem~\ref{thm:altogether}]
(\ref{it:nostates}) We suppose that $\T C^*(E)$ has a KMS$_\beta$ state $\phi$, and 
prove that $H_\beta \not= E^0$. The set $\bigcup\{C \in {E^0/\!\!\sim} :
\ln\rho(A_C)>\beta\}$ generates $H_\beta$ as a hereditary set, and so 
Proposition~\ref{old2.5} implies that $\phi(p_w) = 0$ for all $w \in H_\beta$. Hence 
$1 = \phi(1) = \sum_{v \not\in H_\beta} \phi(p_v)$, and $H_\beta$ cannot be all of 
$E^0$.

(\ref{it:nocrit}) Applying Corollary~\ref{defphiv2} with $H=H_\beta$ gives the existence of the state $\phi_{\beta,v}$, and the last comment in Corollary~\ref{defphiv2} implies that $\phi_{\beta,v}=\phi^{H_\beta}_v\circ q_{H_\beta}$. We can apply Proposition~\ref{defphiv} with $H=H_\beta$, and deduce that the states $\phi^{H_\beta}_v$ for $v \notin H_\beta$ are the extreme points of the KMS$_\beta$ simplex of $\T C^*(E \backslash H_\beta)$. Since $H_\beta$ is not all of $E^0$, the final statement of Proposition~\ref{old2.5} implies that $q^*_{H_\beta}$ is an isomorphism of the KMS$_\beta$ simplex of $\T C^*(E \backslash H_\beta)$ onto that of 
$\T C^*(E)$. Hence the states $\phi_{\beta,v}=\phi^{H_\beta}_v \circ q_{H_\beta}$ are the extreme points of the KMS$_\beta$ simplex of $\T C^*(E)$. 

Lemma~\ref{factorthruC*} implies that $\phi_{\beta,v}=\phi^{H_\beta}_v\circ q_{H_\beta}$ factors through $C^*(E)$ if and only if $v$ is a source in $E \backslash \Sigma H_\beta$. So Lemma~\ref{lem:extremeptsfactor} implies that a KMS$_\beta$ state $\phi$ factors through $C^*(E)$ if and only if it belongs to the convex hull of $\{\phi_{\beta,v}: v \text{ is a source in } E \backslash \Sigma H_\beta\}$.

(\ref{it:critstates}) We can apply Corollary~\ref{defphiv2} with $H=K_\beta$ to get the state $\phi_{\beta,v}=\phi^{K_\beta}\circ q_{K_\beta}$.  As in \eqref{it:nocrit}, $q^*_{H_\beta}$ is an isomorphism of the KMS$_\beta$ simplex of $\T C^*(E \backslash H_\beta)$ onto that of $\T C^*(E)$. Since $\beta = \ln\rho(A_{E^0 \backslash H_\beta})$ is real, $\rho(A_{E^0 \backslash H_\beta})$ cannot be $0$, and \cite[Lemma~A.1(b)]{aHLRS1} implies that $E \backslash H_\beta$ has at least one cycle.  The set $K_\beta\backslash H_\beta$ is generated as a hereditary subset of $E^0\backslash H_\beta$ by 
the minimal critical components of $E \backslash H_\beta$, and hence is the set $H$ in Theorem~\ref{KMScrit} for the graph $E \backslash H_\beta$. Thus Corollary~\ref{cor:state 
decomp} implies that the KMS$_\beta$ states of 
$\T C^*(E \backslash H_\beta)$ have the form $\phi_{r, \epsilon, t}$, and that 
the extreme points are the ones of the form $\phi_{1, \epsilon^v, t} = \phi^{K_\beta\backslash H_\beta}_v \circ q_{K_\beta\backslash H_\beta}=\phi_{\beta,v}$ or $\phi_{0, \epsilon, \delta_C} = \psi^{H_\beta}_C$.
Proposition~\ref{quotmapH} implies that $q_{K_\beta\backslash H_\beta} \circ q_{H_\beta} = q_{K_\beta}$. Thus the KMS$_\beta$ simplex of $\T C^*(E)$ is the convex hull of the set~\eqref{eq:extreme points}. 

It remains to show that a convex combination of the states~\eqref{eq:extreme points} factors through $C^*(E)$ if and only if it belongs to the convex hull of the set~\eqref{convhull}. Lemma~\ref{factorthruC*} implies that $\phi^{K_\beta}_v 
\circ q_{K_\beta}$ factors through $C^*(E)$ if and only if $v$ is a 
source in $E \backslash \Sigma K_\beta$. We claim that the $\psi^{H_\beta}_C 
\circ q_{H_\beta}$ all factor through $C^*(E)$. To see this, fix $C \in \mc(E 
\backslash H_\beta)$. Theorem~\ref{KMScrit}(\ref{crita}) implies that $\psi^{H_\beta}_C$ factors through $C^*(E \backslash H_\beta)$. We have $v E^n C \not= \emptyset$ for all $v \in C$ and $n \in \NN$ because $C$ is a nontrivial connected component. Since $C \cap H_\beta = \emptyset$, we deduce that $C$ does not intersect any of the sets $S_k 
H_\beta$ of~\eqref{eq:satdef}, and hence $C \cap \Sigma H_\beta = 
\emptyset$. Then because $\Sigma H_\beta$ is hereditary, we have $w E^* C = \emptyset$ for all $w \in \Sigma H_\beta$. Hence~\eqref{eq:zCseries} implies that $z^C_w = 0$ for all $w \in \Sigma H_\beta \backslash H_\beta$, and so \eqref{formphiC} implies that $\psi^{H_\beta}_C(p_w) = 0$ for all $w \in \Sigma H_\beta 
\backslash H_\beta$. Now Lemma~\ref{lem:smallerC*E} implies that $\psi_C=\psi^{H_\beta}_C \circ q_{H_\beta}$ 
factors through $C^*(E)$.
\end{proof}

Theorem~\ref{thm:altogether} describes the KMS$_\beta$ simplex for each fixed $\beta$. However, it also makes sense to fix a vertex $v$, and ask for which $\beta$ there is a state $\phi_{\beta,v}$ of $(\T C^*(E),\alpha)$ as in Corollary~\ref{defphiv2}. 

\begin{cor}\label{betav}
Suppose that $E$ is a finite directed graph and $v\in E^0$. Define 
\[
\beta_v:=\max\{\ln\rho(A_C):C\leq v\}.
\]
Then there is a state $\phi_{\beta,v}$ satisfying \eqref{defphibetav} if and only if $\beta>\beta_v$.
\end{cor}

\begin{proof}
First suppose that there exists such a state $\phi_{\beta,v}$. Then there is a hereditary set $H$ such that $v\notin H$ and $\ln\rho(A_{E^0\backslash H})<\beta$. But then any $C$ with $C\leq v$ lies in $E^0\backslash H$, and $\ln\rho(A_C)\leq \ln\rho(A_{E^0\backslash H})<\beta$. Thus $\beta_v=\max\{\ln\rho(A_C):C\leq v\}<\beta$. Conversely, suppose that $\beta>\beta_v$. Then the hereditary closure $K_\beta$ of $\{C:\ln\rho(A_C)\geq \beta\}$ does not contain $v$: for if so, then there exists $C\in \{C:\ln\rho(A_C)\geq \beta\}$ with $C\leq v$, and we have $\beta_v\geq \beta$. Thus we can apply Corollary~\ref{defphiv2} with $H=K_\beta$ to deduce the existence of $\phi_{\beta,v}$.
\end{proof}

\section{Examples}\label{sec:exs}

We give some examples to show how we can use Theorem~\ref{thm:altogether}  to compute all the KMS states on $\T C^*(E)$ and $C^*(E)$. Since we want to focus on how the different components of $E$ interact, we consider graphs in which the components are small.

\begin{ex}\label{dumbbell1}
The following graph $E$
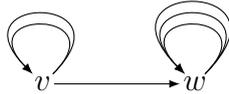
\[
\begin{tikzpicture}
    \node[inner sep=1pt] (v) at (0,0) {$v$};
    \node[inner sep=1pt] (w) at (2,0) {$w$};
    \draw[-latex] (v)--(w);
    \foreach \x in {0,2} {
        \draw[-latex] (v) .. controls +(1.\x,1.\x) and +(-1.\x,1.\x) .. (v);
    }
    \foreach \x in {0,2,4} {
        \draw[-latex] (w) .. controls +(1.\x,1.\x) and +(-1.\x,1.\x) .. (w);
    }
\end{tikzpicture}
\]
has two strongly connected components $\{v\}$ and $\{w\}$. Both are nontrivial components, with $A_{\{v\}}=(2)$, $A_{\{w\}}=(3)$ and $\rho(A)=3$. 

\begin{itemize}
\item For $\beta>\ln\rho(A)=\ln 3$, the set $H_\beta$ of Theorem~\ref{thm:altogether} is empty, and Theorem~\ref{thm:altogether}\eqref{it:nocrit} gives a $1$-dimensional simplex of KMS$_{\beta}$ states on $(\T C^*(E),\alpha)$ with extreme points $\phi_{\beta,v}$ and $\phi_{\beta,w}$. None of these factor through $C^*(E)$. 

\item At $\beta=\ln 3$, $H_{\beta}$ is still empty, but $K_{\ln 3}$ is the hereditary closure of $\{w\}$, which is all of $E^0$. The only critical component is $\{w\}$, and hence Theorem~\ref{thm:altogether}\eqref{it:critstates} gives a  unique KMS$_{\ln 3}$ state  $\psi_{\{w\}}$ which factors through $C^*(E)$. 

\item For $\beta<\ln 3$, $H_\beta=E^0$, and $(\T C^*(E),\alpha)$ has no KMS$_\beta$ states.
\end{itemize}
\end{ex}

\begin{ex}\label{dumbbell2}
Reversing the horizontal arrow in the previous example makes a big difference. The graph $E$ now looks like
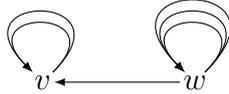
\[
\begin{tikzpicture}
    \node[inner sep=1pt] (v) at (0,0) {$v$};
    \node[inner sep=1pt] (w) at (2,0) {$w$};
    \draw[-latex] (w)--(v);
    \foreach \x in {0,2} {
        \draw[-latex] (v) .. controls +(1.\x,1.\x) and +(-1.\x,1.\x) .. (v);
    }
    \foreach \x in {0,2,4} {
        \draw[-latex] (w) .. controls +(1.\x,1.\x) and +(-1.\x,1.\x) .. (w);
    }
\end{tikzpicture}
\]
The strongly connected components are still $\{v\}$ and $\{w\}$, but now the minimal critical component $\{w\}$ is hereditary. 
\begin{itemize}
\item For $\beta>\ln 3=\ln\rho(A)$, $H_\beta=\emptyset$, and Theorem~\ref{thm:altogether}\eqref{it:nocrit} gives a $1$-dimensional simplex of KMS$_{\beta}$ states on $(\T C^*(E),\alpha)$ with extreme points $\phi_{\beta,v}$ and $\phi_{\beta,w}$. None of these factor through $C^*(E)$. 

\item For $\beta=\ln 3$, we have $H_\beta=\emptyset$ and $K_\beta=\{w\}$. Theorem~\ref{thm:altogether}\eqref{it:critstates} gives a $1$-dimensional simplex of KMS$_{\ln 3}$ states on $(\T C^*(E),\alpha)$ with extreme points $\phi_{\ln 3,v}$ and $\psi_{\{w\}}$, and only $\psi_{\{w\}}$ factors through $C^*(E)$. (We work out a formula for $\psi_{\{w\}}$ at the end of this example.)

\item For $\ln 2<\beta<\ln 3$, $H_\beta=\{w\}$, and Theorem~\ref{thm:altogether}\eqref{it:nocrit} gives a single KMS$_\beta$ state $\phi_{\beta, v}$ on $(\T C^*(E),\alpha)$, which does not factor through $C^*(E)$.

\item For $\beta=\ln 2$, $H_\beta=\{w\}$ and $K_\beta=\{v,w\}=E^0$. The graph $E\backslash H_\beta$ has a single critical component $\{v\}$, and Theorem~\ref{thm:altogether}\eqref{it:critstates} gives a unique KMS$_{\ln 2}$ state $\psi_{\{v\}}$ on $(\T C^*(E),\alpha)$. This state factors through $C^*(E)$.

\item For $\beta<\ln 2$, there are no KMS$_\beta$ states.
\end{itemize}
We can make the construction of these states quite explicit. We illustrate by working through the construction of the KMS$_{\ln 3}$ state $\psi_{\{w\}}$. The unimodular Perron-Frobenius eigenvector for the matrix $A_{\{w\}}=(3)$ is the scalar $x^{\{w\}}_w=1$, and the vector $z^{\{w\}}$ in \eqref{defz} is the scalar
\[
z^{\{w\}}_v=\rho(A)^{-1}\big(1-\rho(A)^{-1}A_{\{v\}}\big)^{-1}A(v,w)x^{\{w\}}_w=3^{-1}(1-3^{-1}.2)^{-1}1.1=3^{-1}.3=1.
\]
Thus $\|1+z^{\{w\}}\|_1=2$, $\psi_{\{w\}}(p_v)=\psi_{\{w\}}(p_w)=2^{-1}$, and
\[
\psi_{\{w\}}(s_\mu s_\nu^*)=\delta_{\mu,\nu}3^{-|\mu|}2^{-1}\quad\text{for $\mu,\nu\in E^*$.}
\]
\end{ex}

\begin{ex}\label{ex:2vertexcomponent}
We now replace the component $\{w\}$ with a 2-vertex component whose critical inverse
temperature still exceeds that of the component $\{v\}$. This gives an example in which the KMS$_\beta$ simplex changes dimension both as $\beta$ decreases to $\ln \rho(A)$, and as $\beta$ passes through $\ln \rho(A)$.
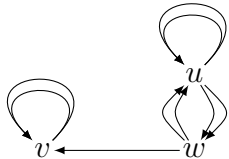
\[
\begin{tikzpicture}
    \node[inner sep=1pt] (v) at (0,0) {$v$};
    \node[inner sep=1pt] (w) at (2,0) {$w$};
    \node[inner sep=1pt] (u) at (2,1) {$u$};
    \draw[-latex] (w)--(v);
    \foreach \x/\xx in {0/3,2/5} {
        \draw[-latex] (v) .. controls +(1.\x,1.\x) and +(-1.\x,1.\x) .. (v);
        \draw[-latex] (w) .. controls +(-0.\xx,0.5) .. (u);
        \draw[-latex] (u) .. controls +(0.\xx,-0.5) .. (w);
        \draw[-latex] (u) .. controls +(1.\x,1.\x) and +(-1.\x,1.\x) .. (u);
    }
\end{tikzpicture}
\]
The strongly connected components are $\{v\}$ and  $\{w,u\}$. The block corresponding to the latter is $A_{\{w,u\}} =
\big(\begin{smallmatrix}2&2\\2&0\end{smallmatrix}\big)$, which has spectral radius  $\rho(A_{\{w,u\}})=\gamma := 1 + \sqrt5$. Since $\rho(A)=\max\{\rho(A_{\{v\}}),\rho(A_{\{w,u\}})\}=\gamma$, $\{w,u\}$ is a minimal critical component.
\begin{itemize}
\item For $\beta>\ln \gamma$, $H_\beta=\emptyset$, and
    Theorem~\ref{thm:altogether}\eqref{it:nocrit} gives a $2$-dimensional simplex of
    KMS$_{\beta}$ states on $(\T C^*(E),\alpha)$ with extreme points $\phi_{\beta,
    v}$, $\phi_{\beta, w}$ and $\phi_{\beta,u}$. None of these factor through
    $C^*(E)$.

\item For $\beta = \ln\gamma$, we have $H_\beta = \emptyset$ and $K_\beta =
    \{w,u\}$. 
    Theorem~\ref{thm:altogether}\eqref{it:critstates} gives a $1$-dimensional simplex
    of KMS$_{\ln\gamma}$ states on $(\T C^*(E),\alpha)$ with extreme points
    $\phi_{\ln\gamma,v}$ and $\psi_{\{w,u\}}$. Only $\psi_{\{w,u\}}$ factors through
    $C^*(E)$.

\item For $0<\beta<\ln\gamma$, we have $\{w,u\} \subseteq H_\beta$, and so the KMS$_\beta$ simplex is similar to that of Example~\ref{dumbbell2}. In particular, the
    dimension of the KMS$_\beta$ simplex drops again to 0 as $\beta$ drops below
    $\ln \gamma$, and disappears altogether for $\beta<\ln 2$.
\end{itemize}
\end{ex}

\begin{ex}\label{EminusGsourced} In the next graph $E$, we have added two trivial components, and now the subtleties involving saturations in Theorem~\ref{thm:altogether} come into play.
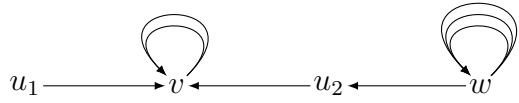
\[
\begin{tikzpicture}
\node[inner sep=1pt] (u_1) at (-2,0) {$u_1$};
\node[inner sep=1pt] (u_2) at (2,0) {$u_2$};
    \node[inner sep=1pt] (v) at (0,0) {$v$};
    \node[inner sep=1pt] (w) at (4,0) {$w$};
    \draw[-latex] (w)--(u_2);
     \draw[-latex] (u_2)--(v);
     \draw[-latex] (u_1)--(v);
    \foreach \x in {0,2} {\draw[-latex] (v) .. controls +(1.\x,1.\x) and +(-1.\x,1.\x) .. (v);}
    \foreach \x in {0,2,4} {
        \draw[-latex] (w) .. controls +(1.\x,1.\x) and +(-1.\x,1.\x) .. (w);
    }
\end{tikzpicture}
\]
\begin{itemize}
\item For $\beta>\ln 3$, we have $H_\beta=\emptyset$, and Theorem~\ref{thm:altogether}\eqref{it:nocrit} gives us a $3$-dimensional simplex of KMS$_\beta$ states on $(\T C^*(E),\alpha)$ with extreme points $\phi_{\beta,v}$, $\phi_{\beta,w}$, $\phi_{\beta,u_1}$ and $\phi_{\beta,u_2}$. The state $\phi_{\beta,u_1}$ factors through $C^*(E)$.

\item At $\beta=\ln 3$, we have $H_\beta=\emptyset$ and $K_\beta=\{w\}$. Theorem~\ref{thm:altogether}\eqref{it:critstates} gives us a $3$-dimensional simplex of KMS$_{\ln 3}$ states, with extreme points $\phi_{\ln 3,v}$, $\phi_{\ln 3,u_1}$ and $\phi_{\ln 3,u_2}$ alongside the state $\psi_{\{w\}}$ associated to the critical component $\{w\}$ in $K_\beta$. Now $\Sigma K_\beta=\{u_2, w\}$, and the vertex $u_1$ is a source in $E\backslash \Sigma K_\beta$. Thus both $\psi_{\{w\}}$ and $\phi_{\ln 3,u_1}$ factor through KMS$_{\ln 3}$ states of $(C^*(E),\alpha)$.

\item For $\ln 2<\beta<\ln 3$, we have $H_\beta=\{w\}$, and Theorem~\ref{thm:altogether}\eqref{it:nocrit} gives us a $2$-dimensional simplex of KMS$_\beta$ states on $(\T C^*(E),\alpha)$ with extreme points $\phi_{\beta,v}$, $\phi_{\beta,u_1}$ and $\phi_{\beta,u_2}$. Since $\Sigma H_\beta=\{u_2,w\}$, only the state $\phi_{\beta, u_1}$ factors through $C^*(E)$.

\item For $\beta=\ln 2$, we have $H_\beta=\{w\}$ and $K_\beta=E^0$. The only critical component in $E\backslash H_\beta$ is $\{v\}$, and hence Theorem~\ref{thm:altogether}\eqref{it:critstates} implies that $(\T C^*(E),\alpha)$ has a unique KMS$_{\ln 2}$ state $\psi_{\{v\}}$, and that this state factors through $C^*(E)$. 

\item For $\beta<\ln 2$, the hereditary closure of $H_\beta=\{v,w\}$ is all of $E^0$, and $(\T C^*(E),\alpha)$ has no KMS$_\beta$ states.
\end{itemize}
\end{ex}

\begin{ex}\label{ex4}
Our next graph $E$ is the one from Example~\ref{EminusGsourced} with the edge between $u_1$ and $v$ reversed.
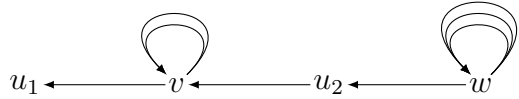
\[
\begin{tikzpicture}
\node[inner sep=1pt] (u_1) at (-2,0) {$u_1$};
\node[inner sep=1pt] (u_2) at (2,0) {$u_2$};
    \node[inner sep=1pt] (v) at (0,0) {$v$};
    \node[inner sep=1pt] (w) at (4,0) {$w$};
    \draw[-latex] (w)--(u_2);
     \draw[-latex] (u_2)--(v);
     \draw[-latex] (v)--(u_1);
    \foreach \x in {0,2} {\draw[-latex] (v) .. controls +(1.\x,1.\x) and +(-1.\x,1.\x) .. (v);}
    \foreach \x in {0,2,4} {
        \draw[-latex] (w) .. controls +(1.\x,1.\x) and +(-1.\x,1.\x) .. (w);
    }
\end{tikzpicture}
\]
\begin{itemize}
\item For $\beta> \ln 3$ and $\beta=\ln 3$, we still have a $3$-dimensional simplex of KMS$_{\beta}$ states on $(\T C^*(E),\alpha)$. However, for this graph $u_1$ is not a source in $E\backslash \Sigma K_{\ln 3}=E^0\backslash \{u_2,w\}$, and only the KMS$_{\ln 3}$ state $\psi_{\{w\}}$ factors through $C^*(E)$.

\item For $\ln 2<\beta<\ln 3$, we still have $H_\beta=\{w\}$ and a 2-dimensional simplex of KMS$_\beta$ states. For this graph, none of these KMS states factors through $C^*(E)$.

\item At $\beta=\ln 2$, $K_{\beta}=\{v,u_2,w\}$, and we have a $1$-dimensional simplex of KMS$_{\ln 2}$ states on $(\T C^*(E),\alpha)$ with extreme points $\psi_{\{u_2,w\}}$ and $\phi_{\ln 2,u_1}$. The state $\psi_{\{u_2,w\}}$ factors through $C^*(E)$.

\item For $\beta<\ln 2$, we have $H_\beta=\{v, u_2,w\}$, and  a single KMS$_\beta$ state $\phi_{\beta,u_1}$ on $(\T C^*(E),\alpha)$. Since $\Sigma H_\beta$ is all of $E^0$, this state does not factor through $C^*(E)$.
\end{itemize}
\end{ex}

\begin{ex}\label{ex:persistentsource}
We now add a source $u_3$ to the graph of Example~\ref{ex4}.
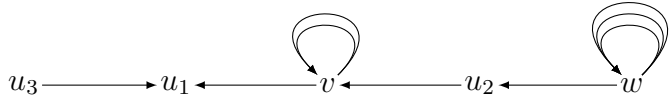
\[
\begin{tikzpicture}
    \node[inner sep=1pt] (u_3) at (-4,0) {$u_3$};
    \node[inner sep=1pt] (u_1) at (-2,0) {$u_1$};
    \node[inner sep=1pt] (u_2) at (2,0) {$u_2$};
    \node[inner sep=1pt] (v) at (0,0) {$v$};
    \node[inner sep=1pt] (w) at (4,0) {$w$};
    \draw[-latex] (u_3)--(u_1);
    \draw[-latex] (w)--(u_2);
    \draw[-latex] (u_2)--(v);
    \draw[-latex] (v)--(u_1);
    \foreach \x in {0,2} {\draw[-latex] (v) .. controls +(1.\x,1.\x) and +(-1.\x,1.\x) .. (v);}
    \foreach \x in {0,2,4} {
        \draw[-latex] (w) .. controls +(1.\x,1.\x) and +(-1.\x,1.\x) .. (w);
    }
\end{tikzpicture}
\]
The vertex $u_3$ belongs to the complement of $H_\beta$ and of $K_\beta$ for all $\beta$.
So at every $\beta$, the new vertex $u_3$ gives an extreme point
$\phi_{\beta, u_3}$ of the KMS$_\beta$ simplex of $\T C^*(E)$, and this state factors through $C^*(E)$.
\end{ex}

\begin{ex}
The following graph $E$
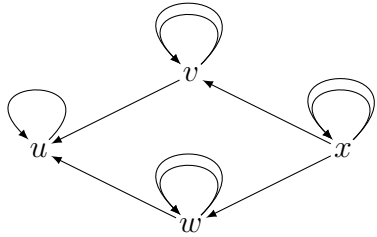
\[
\begin{tikzpicture}
    \node[inner sep=1pt] (u) at (0,0) {$u$};
    \node[inner sep=1pt] (v) at (2,1) {$v$};
    \node[inner sep=1pt] (w) at (2,-1) {$w$};
    \node[inner sep=1pt] (x) at (4,0) {$x$};
    \draw[-latex] (x)--(v);
    \draw[-latex] (x)--(w);
    \draw[-latex] (w)--(u);
    \draw[-latex] (v)--(u);
    \foreach \x in {0,2} {
        \draw[-latex] (x) .. controls +(1.\x,1.\x) and +(-1.\x,1.\x) .. (x);
    }
    \foreach \x in {0,2} {
        \draw[-latex] (v) .. controls +(1.\x,1.\x) and +(-1.\x,1.\x) .. (v);
    }
    \foreach \x in {0,2} {
        \draw[-latex] (w) .. controls +(1.\x,1.\x) and +(-1.\x,1.\x) .. (w);
    }
    \foreach \x in {0} {
        \draw[-latex] (u) .. controls +(1.\x,1.\x) and +(-1.\x,1.\x) .. (u);
    }
\end{tikzpicture}
\]
has three components $C$ with $\rho(A)=\rho(A_C)$, but only $\{v\}$ and $\{w\}$ are minimal. 
\begin{itemize}
\item For $\beta>\ln 2$, we have a $3$-dimensional simplex of KMS$_\beta$ states on $(\T C^*(E),\alpha)$, and none of them factor through $C^*(E)$. 

\item At $\beta=\ln 2$, we have a $2$-dimensional simplex of KMS$_{\ln 2}$ states on $(\T C^*(E),\alpha)$ with extreme points $\psi_{\{v\}}$, $\psi_{\{w\}}$ and $\phi_{\beta, u}$. Of these, only $\psi_{\{v\}}$ and $\psi_{\{w\}}$ factor through $C^*(E)$. 

\item For $0\leq \beta<\ln 2$, there is a unique KMS$_\beta$ state on $\T C^*(E)$, which only factors through $C^*(E)$ when $\beta=0$. This KMS$_0$ state is the invariant trace on $C^*(E)$ that is obtained by lifting the trace on $C^*(E\backslash \Sigma H)\cong C(\TT)$ given by integration against Haar measure on~$\TT$.
\end{itemize}
\end{ex}

\section{Concluding Remarks}\label{CR}

\subsection{Critical inverse temperatures}\label{CIT}
We say that $\beta$ is a \emph{critical inverse temperature} if $H_\beta\not= E^0$ and $\beta=\ln\rho(A_{E^0\backslash H_\beta})$. Theorem~\ref{thm:altogether}\eqref{it:critstates} says that these are precisely the inverse temperatures at which we have states of the form $\psi_C$, and that these states factor through $C^*(E)$; for all but the smallest critical $\beta$, we also have states of the form $\phi_{\beta,v}$. 

Every critical inverse temperature $\beta$ has the form $\ln\rho(A_C)$ for some component $C$, but as our examples show, not every $\ln\rho(A_C)$ need be critical. (For example, $\beta=\ln 2$ in Example~\ref{dumbbell1}.) So to find the critical $\beta$ for a given finite graph $E$, we compute the numbers $\beta=\ln\rho(A_C)$,  identify the sets $H_\beta$ by looking at the graph, and discard the numbers which are not critical. The set of critical inverse temperatures is always finite (with cardinality bounded by $|E^0|$), but could in general be arbitarily large.

Since there are finitely many critical values, we can list them in increasing order. Then for $\beta$ between two consecutive critical values, say $\beta\in(\beta_C,\beta_D)$, Theorem~\ref{thm:altogether}\eqref{it:nocrit} gives a simplex of KMS$_\beta$ states with extreme points $\{\phi_{\beta,v}:v\in E^0\backslash H_\beta\}$.

 For the Toeplitz algebra, the range of possible inverse temperatures $\beta$ is either $\RR$ (if $E$ has a source which does not talk to any nontrivial component $C$) or $[\beta_l,\infty)$, where $\beta_l$ is the smallest critical inverse temperature. But for the graph algebra $C^*(E)$ there are interesting number-theoretic restrictions on the possible values of critical $\beta$ and thus on the range of possible inverse temperatures. We use results of Lind \cite{Lin}, and refer to the treatment in \cite[\S11.1]{Lin-Mar}.

Suppose that $E$ is a directed graph without sources and suppose that $C^*(E)$ has a KMS$_\beta$ state. Then $\beta=\ln\rho(A_C)$ for some strongly connected component $C$ of $E$. Since $\rho(A_C)$ is the Perron-Frobenius eigenvalue of $A_C$, it is a root of the characteristic polynomial $\det(x1-A_C)$, which is a monic polynomial of degree $n$ with integer coefficients. Thus $\rho(A_C)$ is an algebraic integer. For each algebraic integer $\lambda$ there is a unique \emph{minimal polynomial} $q_\lambda(x)\in \QQ[x]$ that is monic, irreducible and has $q_\lambda(\lambda)=0$ \cite[Proposition~6.1.7]{IR}; the other roots of this polynomial are called the \emph{conjugates} of $\lambda$. A \emph{Perron number} is an algebraic integer $\lambda \geq 1$ that is strictly larger than the absolute value of all its other conjugates.

 \begin{prop}\label{excPerron} 
Suppose that $\beta> 0$. Then $e^{p\beta}$ is a Perron number for some $p\in \NN$ if and only if there exists a graph $E$ without sources such that the gauge dynamics on $C^*(E)$ has a KMS$_\beta$ state.
 
 \end{prop}
\begin{proof} Let $E$ be a graph such that $C^*(E)$ has a KMS$_\beta$ state, and choose a component $C$ such that $\beta=\ln\rho(A_C)$, as above. Let $p$ be the period of the irreducible matrix $A_C$, then $(e^{\beta})^p=e^{p\beta}$ is a Perron number by the implication 
$(1)\Longrightarrow(3)$ of \cite[Theorem 11.1.5]{Lin-Mar}.  Conversely, if $e^{p\beta}$ is a Perron number for some $p\in \NN$, the implication $(3)\Longrightarrow (1)$ of the same theorem gives the existence of a nonnegative integer matrix $A$ with spectral radius~$e^\beta$. Thus for the graph $E$ with vertex matrix $A$, $C^*(E)$ has a KMS$_\beta$ state.
\end{proof}

It is easy to produce Perron numbers and also algebraic integers $\lambda \geq 1$ that are not Perron numbers. For example, $(5-\sqrt 5)/2$ is an algebraic integer with minimal polynomial $x^2-5x+5$, and hence the conjugates are $(5\pm \sqrt 5)/2$. Thus Proposition~\ref{excPerron} implies that there is no graph without sources such that $C^*(E)$ has a KMS state with inverse temperature $\ln((5-\sqrt 5)/2)$. Note that $x^2-5x+5$ is the characteristic polynomial of 
\[
A=\begin{pmatrix}3&1\\1&2
\end{pmatrix},
\]
which is the vertex matrix of a graph with two vertices.

\subsection{Connections with the results of Carlsen and Larsen}
In their recent preprint \cite{CL}, Carlsen and Larsen  study the KMS states of generalized gauge dynamics on the relative graph algebras of possibly infinite graphs using the partial action techniques developed by Exel and Laca in \cite{EL}. Their results apply in particular to finite graphs\footnote{Though in \cite{CL} they use the non-functorial convention for paths in directed graphs, so strictly speaking one would have to apply their results to the opposite graph $E^{\text{opp}}=(E^0,E^1,s,r)$.}, where taking their function $N:E^1\to \RR$ to be identically $1$ gives the action $\alpha:\RR\to \Aut \T C^*(E)$ studied here. 

To make the connection, we observe that the sum  $y_v := \sum_{\mu\in E^*v} e^{-\beta |\mu|}$ in \cite[Theorem~3.1]{aHLRS1} is the same as that defining the ``fixed-target partition function'' $Z_v(\beta)$ in \cite[Equation (5.8)]{CL} (see also \cite[Definition 9.3]{EL}). Our results allow us to identify the intervals of convergence of these partition functions:

\begin{lem}
 Let $E$ be a finite graph. For $v\in E^0$, take $\beta_v
=\max\{\ln\rho(A_C):C\leq v\}$ as in Corollary~\ref{betav}. Then the series 
$\sum_{\mu\in E^*v} e^{-\beta |\mu|}$ converges if and only if $\beta>\beta_v$. 
\end{lem} 

\begin{proof}
If $\beta>\beta_v$, then taking $H=H_{\beta_v}$ in Proposition~\ref{defphiv} shows that the series converges. On the other hand, suppose that $\beta \leq \beta_v$ and $C$ is a component such that $\beta \leq \beta_C$.
Choose a path  $\lambda$ in $CE^*v$. Then 
\[
\sum_{\mu\in E^*v} e^{-\beta |\mu|}  \geq e^{-\beta|\lambda|} \sum_{\mu'\in CE^*r(\lambda)} e^{-\beta |\mu'|},
\]
and the series $\sum_{\mu'\in CE^*r(\lambda)} e^{-\beta |\mu'|}$ diverges because $\rho(e^{-\beta}A_C )\geq 1$ and $A_C$ is irreducible. Thus  $\sum_{\mu\in E^*v} e^{-\beta|\mu|}$ diverges too. 
\end{proof}

The factors $\delta_{\mu,\nu}$ in our formulas for the values of KMS states show that all the KMS states on $\T C^*(E)$ and $C^*(E)$ factor through the expectation onto the diagonal $D:=\clsp\{s_\lambda s_\lambda^*:\lambda\in E^*\}$. The restriction to $D$ is then given by a measure $\nu$ on the spectrum of $D$, which is $E^*\cup E^\infty$. Exel and Laca say that a KMS state $\psi$ is of \emph{finite type} if this measure $\nu$ is supported on the set $E^*$ of finite paths, and of \emph{infinite type} if $\nu$ is supported on the set $E^\infty$ of infinite paths \cite{EL}. (These are described as \emph{infinite type (A)} in \cite{CL}; for finite $E$, $E^\infty$ has no wandering infinite paths, and hence there are no states which are of their infinite type (B).) 

The states $\phi_{\beta,v}$ have the form $\phi_\epsilon\circ q_H$ with $\epsilon$ a point mass supported at $v$. In \cite[\S6.4]{aHLRS1} we described measures on $E^*$ for the states of the form $\phi_\epsilon$, so they and the $\phi_{\beta,v}$ are of finite type. The states $\psi_C$ factor through $C^*(E)$, and  are of infinite type; to see this\footnote{To see what goes wrong with this argument for states $\phi_{\beta,v}$ which factor through $C^*(E)$, notice that then Theorem~\ref{thm:altogether} implies that $v$ is a source in some $E\backslash \Sigma H_\beta$ or $E\backslash \Sigma K_\beta$, and hence $v$ is also a source in $E$. But then no Cuntz-Krieger relation is imposed at $v$ in $C^*(E)$, so \eqref{compmeas} is not valid.}, we use that $\psi_C$ factors through $C^*(E)$ to compute
\begin{equation}\label{compmeas}
\nu(\{\lambda\})=\nu(Z(\lambda))-\sum_{r(e)=s(\lambda)}\nu(Z(\lambda e))=\psi_C(s_\lambda s^*_\lambda)-\sum_{r(e)=s(\lambda)}\psi_C(s_{\lambda e}s^*_{\lambda e})=0.
\end{equation}
Thus for $\beta$ between two critical values, say $\beta\in (\beta_C,\beta_D)$, the set $\ebetareg$ in \cite[Definition~5.5]{CL} is $E\backslash H_{\beta_D}$ and for $\beta$ critical it is $E\backslash K_\beta$. The set $\ebetacri$ is empty unless $\beta$ is critical, and it is the union of the criticial components in $E\backslash H_\beta$ if $\beta$ is critical. If $\beta_C$ is critical and there are sources in $E\backslash \Sigma K_{\beta_C}$, then $(C^*(E),\alpha)$ has KMS$_{\beta_C}$ states of both finite and infinite type.

\end{document}